\documentclass[oneside,reqno,11pt]{amsart}

\usepackage{geometry}
\usepackage{mathtools}
\usepackage{amsmath}
\usepackage{amsthm}
\usepackage{amssymb}
\usepackage{tikz}
\usepackage{tikz-cd}
\usepackage[all,cmtip]{xy}
\usepackage{float}
\usepackage{indentfirst}
\usepackage[mathscr]{euscript}
\usepackage{stmaryrd}
\usepackage{hyperref}
\usepackage{csquotes}
\usepackage{multirow}
\usepackage{verbatim}

\makeatletter
\def\blfootnote{\gdef\@thefnmark{}\@footnotetext}
\makeatother

\theoremstyle{plain}
\newtheorem{thm}{Theorem}[section]
\newtheorem{prop}[thm]{Proposition}
\newtheorem{lem}[thm]{Lemma}
\newtheorem{cor}[thm]{Corollary}
\newtheorem{conj}[thm]{Conjecture}

\theoremstyle{definition}
\newtheorem{dfn}[thm]{Definition}
\newtheorem{ex}[thm]{Example}
\newtheorem*{ack}{Acknowledgements}

\theoremstyle{remark}
\newtheorem{rmk}[thm]{Remark}

\numberwithin{equation}{section}

\title[On Gromov--Yomdin type theorems]{On Gromov--Yomdin type theorems and a categorical interpretation of holomorphicity}

\author{Federico Barbacovi}
\address{Department of Mathematics, University College London, Gower Street, London WC1E 6BT, United Kingdom}
\email{federico.barbacovi.18@ucl.ac.uk}

\author{Jongmyeong Kim}
\address{Center for Geometry and Physics, Institute for Basic Science (IBS), Pohang 37673, Republic of Korea}
\email{myeong@ibs.re.kr}

\begin{document}

\begin{abstract}
In topological dynamics, the Gromov--Yomdin theorem states that the topological entropy of a holomorphic automorphism $f$ of a smooth projective variety is equal to the logarithm of the spectral radius of the induced map $f^*$.
In order to establish a categorical analogue of the Gromov--Yomdin theorem, one first needs to find a categorical analogue of a holomorphic automorphism.
In this paper, we propose a categorical analogue of a holomorphic automorphism and prove that the Gromov--Yomdin type theorem holds for them.
\end{abstract}

\maketitle
\tableofcontents

\section{Introduction}\label{intro}

\subsection{Gromov--Yomdin theorem}

A pair $(X,f)$ of a compact Hausdorff space $X$ and a continuous map $f : X \to X$ is called a {\em topological dynamical system}.
For a topological dynamical system $(X,f)$, one can define its {\em topological entropy} $h_\mathrm{top}(f) \in [0,\infty)$ which measures how complicated the dynamics of $f$ is.

In some cases, the topological entropy of $f$ can be computed or estimated using linear algebraic information about $f$.
Recall that, for a linear map $A$, its {\em spectral radius} $\rho(A)$ is defined as the largest absolute value of the eigenvalues of $A$ (see Definition \ref{spec-rad}).

\begin{thm}[Yomdin's inequality {\cite{Yom}}]\label{yom-thm}
Let $X$ be a compact smooth manifold and $f : X \to X$ be a diffeomorphism.
Then we have
\begin{equation*}
h_\mathrm{top}(f) \geq \log\rho(f^* : H^*(X) \to H^*(X)).
\end{equation*}
\end{thm}

In order to obtain the opposite inequality, we have to put an additional condition on $f$.

\begin{thm}[Gromov's inequality {\cite{Gro}}]\label{gro-thm}
Let $X$ be a smooth projective variety and $f : X \to X$ be a holomorphic automorphism.
Then we have
\begin{equation*}
h_\mathrm{top}(f) \leq \log\rho(f^* : H^*(X) \to H^*(X)).
\end{equation*}
\end{thm}

Combining these two theorems, we obtain the following corollary which we refer to as the {\em Gromov--Yomdin theorem}.

\begin{cor}[Gromov--Yomdin theorem]\label{gy-thm}
Let $X$ be a smooth projective variety and $f : X \to X$ be a holomorphic automorphism.
Then we have
\begin{equation*}
h_\mathrm{top}(f) = \log\rho(f^* : H^*(X) \to H^*(X)).
\end{equation*}
\end{cor}

As a categorical analogue of the notion of a topological dynamical system, Dimitrov--Haiden--Katzarkov--Kontsevich \cite{DHKK} introduced the notion of a {\em categorical dynamical system}: it is a pair $(\mathcal{D},\Phi)$ of a triangulated category $\mathcal{D}$ and an exact endofunctor $\Phi : \mathcal{D} \to \mathcal{D}$.
They also introduced its {\em categorical entropy} $h_t(\Phi)$ which carries a real parameter $t$ (see Definition \ref{cat-ent}).

When $\mathcal{D}$ is the bounded derived category $D^b\mathrm{Coh}(X)$ of a smooth projective variety $X$ and $\Phi$ is the derived pullback $\mathbb{L}f^*$ of a surjective endomorphism $f$ of $X$, Kikuta--Takahashi showed that
\begin{equation*}
h_0(\mathbb{L}f^*) = h_\mathrm{top}(f) = \log\rho([\mathbb{L}f^*] : \mathcal{N}(X) \to \mathcal{N}(X))
\end{equation*}
where $\mathcal{N}(X)$ denotes the {\em numerical Grothendieck group} of $X$.
It is thus natural to ask whether the {\em Gromov--Yomdin type theorem} holds for the categorical entropy.

\begin{conj}[{\cite[Conjecture 5.3]{KT}}]\label{kt-conj}
Let $D^b\mathrm{Coh}(X)$ be the bounded derived category of a smooth projective variety $X$ and $\Phi : D^b\mathrm{Coh}(X) \to D^b\mathrm{Coh}(X)$ be an exact endofunctor.
Then we have
\begin{equation*}
h_0(\Phi) = \log\rho([\Phi] : \mathcal{N}(X) \to \mathcal{N}(X)).
\end{equation*}
\end{conj}

It is now known that Conjecture \ref{kt-conj} is true in many cases \cite{BK1,Fan2,Kik1,KST,KT,Kim,Ouc,Yos} and also that there are counterexamples \cite{BK1,Fan1,Mat,Ouc}.
Therefore it is important to characterize exact endofunctors which do or do not satisfy the Gromov--Yomdin type theorem.

In this paper, we consider more generally a numerically finite triangulated category $\mathcal{D}$ and ask when an exact autoequivalence $\Phi : \mathcal{D} \to \mathcal{D}$ satisfies
\begin{equation*}
h_0(\Phi) = \log\rho([\Phi] : \mathcal{N}(\mathcal{D}) \to \mathcal{N}(\mathcal{D}))
\end{equation*}
where $\mathcal{N}(\mathcal{D})$ denotes the {\em numerical Grothendieck group} of $\mathcal{D}$.

\subsection{Motivation from homological mirror symmetry}

A categorical analogue of Theorem \ref{yom-thm} (Yomdin's inequality) is known to hold quite generally (see Section \ref{yom-sec}).
The problem is to establish a categorical analogue of Theorem \ref{gro-thm} (Gromov's inequality).
In order to do that, we first have to find a categorical analogue of a holomorphic automorphism.
We shall do this through the philosophy of the homological mirror symmetry conjecture \cite{Kon}.

The homological mirror symmetry conjecture states that, for a Calabi--Yau manifold $(X,\omega)$ (viewed as a symplectic manifold), there is another Calabi--Yau manifold $(X^\vee,\Omega^\vee)$ (viewed as a complex manifold) such that
\begin{equation}\label{hms}
D^\pi\mathrm{Fuk}(X,\omega) \simeq D^b\mathrm{Coh}(X^\vee,\Omega^\vee)
\end{equation}
where $D^\pi\mathrm{Fuk}(X,\omega)$ denotes the split-closed derived Fukaya category of $(X,\omega)$.

It is believed that the equivalence \eqref{hms} induces an isomorphism between the {\em complex moduli space} of $(X,\omega)$ and the {\em K\"ahler moduli space} of $(X^\vee,\Omega^\vee)$.
Bridgeland \cite{Bri2} and Joyce \cite{Joy} conjectured that this isomorphism can be understood as an isomorphism between (some quotients of) the spaces of {\em stability conditions}.
Roughly speaking, they conjectured that the space of stability conditions on $D^\pi\mathrm{Fuk}(X,\omega)$ can be described as follows:
\begin{center}
\begin{tabular}{|c|c|}\hline
symplectic manifold $(X,\omega)$ & derived Fukaya category $D^\pi\mathrm{Fuk}(X,\omega)$\\ \hline
holomorphic volume form $\Omega$ & stability condition $\sigma$\\ \hline
special Lagrangian $L$ of phase $\phi$ & semistable object $E$ of phase $\phi$\\ \hline
period integral $\int_L \Omega$ & central charge $Z_\sigma(E)$\\ \hline
\end{tabular}
\end{center}

A choice of a complex structure on $(X,\omega)$ determines a holomorphic volume form (up to scalar).
Now consider a symplectomorphism $f : X \to X$.
Assume that there is a complex structure on $X$ with respect to which $f$ is holomorphic.
Let us fix a holomorphic volume form $\Omega$ determined by that complex structure.
The holomorphicity then implies that $f^* : H^*(X,\mathbb{C}) \to H^*(X,\mathbb{C})$ preserves the Hodge decomposition.
Hence, as $\dim H^{n,0}(X,\mathbb{C}) = 1$, it follows that
\begin{equation}\label{holo}
f^*\Omega = e^{-i\pi\psi}\Omega
\end{equation}
for some $\psi \in \mathbb{R}$.

Let us interpret the condition \eqref{holo} in terms of stability conditions using Bridgeland--Joyce's idea.
Let $f_* : D^\pi\mathrm{Fuk}(X,\omega) \to D^\pi\mathrm{Fuk}(X,\omega)$ be the exact autoequivalence induced by $f$ and $\sigma_\Omega$ be the stability condition on $D^\pi\mathrm{Fuk}(X,\omega)$ induced by $\Omega$.
Then, using the left action of $f_* \in \mathrm{Aut}(D^\pi\mathrm{Fuk}(X,\omega))$ and the right action of $\psi \in \mathbb{R} \subset \mathbb{C}$ on $\sigma_\Omega \in \mathrm{Stab}(D^\pi\mathrm{Fuk}(X,\omega))$, we can interpret the equation \eqref{holo} as
\begin{equation}\label{holo-stab}
f_* \cdot \sigma_\Omega = \sigma_\Omega \cdot \psi.
\end{equation}

In summary, we can categorify one of the properties of a holomorphic automorphism.
We will use this property to formulate the Gromov--Yomdin type theorems for the categorical entropy and its relatives.

\subsection{Main results}

Let us now state our main results.
For the sake of simplicity, we shall only consider numerical stability conditions in this section.
In the main part of the paper, we will consider more generally stability conditions whose central charges factor through some finite lattice (not necessarily the numerical Grothendieck group).

For a numerically finite triangulated category $\mathcal{D}$, let $\mathrm{Aut}(\mathcal{D})$ be the group of exact autoequivalences of $\mathcal{D}$ and $\mathrm{Stab}_\mathcal{N}^\dagger(\mathcal{D})$ be a connected component of the space of numerical stability conditions on $\mathcal{D}$.
Denote by $\widetilde{GL}^+(2,\mathbb{R})$ the universal cover of $GL^+(2,\mathbb{R})$ which consists of elements $(M,f)$ where $M \in GL^+(2,\mathbb{R})$ and $f : \mathbb{R} \to \mathbb{R}$ is an increasing function satisfying $f(\phi+1)=f(\phi)+1$ and $Me^{i\pi\phi} \in \mathbb{R}_{>0}e^{i\pi f(\phi)}$.
The group $\mathrm{Aut}(\mathcal{D})$ (resp. $\widetilde{GL}^+(2,\mathbb{R})$) acts on $\mathrm{Stab}_\mathcal{N}^\dagger(\mathcal{D})$ from the left (resp. right) (see Section \ref{stab-sec}).

In view of the condition \eqref{holo-stab}, we introduce the following notion.

\begin{dfn}[see Definition \ref{triple}]\label{triple-intro}
A triple $(\Phi,\sigma,g) \in \mathrm{Aut}(\mathcal{D}) \times \mathrm{Stab}_\mathcal{N}^\dagger(\mathcal{D}) \times \widetilde{GL}^+(2,\mathbb{R})$ is called {\em compatible} if it satisfies
\begin{equation}\label{comp}
\Phi \cdot \sigma = \sigma \cdot g.
\end{equation}
\end{dfn}

\begin{rmk}
When $g \in \widetilde{GL}^+(2,\mathbb{R})$ is given by
\begin{equation*}
\left( \begin{pmatrix} \cos\pi\psi & -\sin\pi\psi\\ \sin\pi\psi & \cos\pi\psi \end{pmatrix}, \phi \mapsto \phi+\psi \right)
\end{equation*}
for some $\psi \in \mathbb{R}$, the condition \eqref{comp} corresponds to the condition \eqref{holo-stab} and therefore it can be considered as a categorical analogue of the condition \eqref{holo}.
Thus Definition \ref{triple-intro} categorifies (and generalizes) the notion of a holomorphic automorphism.
\end{rmk}

Our first main theorem states that, for such $\Phi$, the Gromov--Yomdin type theorems hold for the mass growth $h_{\sigma,t}(\Phi)$ and the polynomial mass growth $h_{\sigma,t}^\mathrm{pol}(\Phi)$ (see Definitions \ref{mass-grow} and \ref{pol-mass-grow}) under some mild assumption.

Let us say the central charge $Z_\sigma : \mathcal{N}(\mathcal{D}) \to \mathbb{C}$ has {\em spanning image} if the set
\begin{equation*}
\{Z_\sigma(E) \,|\, E \text{ is } \sigma \text{-semistable} \}
\end{equation*}
is an $\mathbb{R}$-spanning set of $\mathbb{C} \cong \mathbb{R}^2$.

\begin{thm}[see Theorem \ref{cat-gy-thm}]\label{main-thm-gy}
Let $(\Phi,\sigma,g) \in \mathrm{Aut}(\mathcal{D}) \times \mathrm{Stab}_\mathcal{N}^\dagger(\mathcal{D}) \times \widetilde{GL}^+(2,\mathbb{R})$ be a compatible triple such that $Z_\sigma$ has spanning image.
Then we have the following:
\begin{enumerate}
\item $h_{\sigma,0}(\Phi) = \log\rho(\Phi) = \log\rho(M_g)$.
\item $h_{\sigma,0}^\mathrm{pol}(\Phi) = s(\Phi) = s(M_g)$.
\end{enumerate}
Here, for a linear map $A$, the {\em polynomial growth} $s(A)$ is given by one less than the maximal size of the Jordan blocks of $A$ whose eigenvalues have absolute value $\rho(A)$ (see Definition \ref{pol-grow} and Lemma \ref{jordan}).
\end{thm}

Assume that $\mathcal{D}$ is moreover {\em saturated}, i.e., it admits a smooth and proper dg enhancement.
In this case, we can define the {\em shifting number} $\overline{\nu}(\Phi)$ and the {\em polynomial shifting number} $\overline{\nu}^\mathrm{pol}(\Phi)$ (see Definitions \ref{shift-num} and \ref{pol-shift-num}).
The following theorem says that if $(\Phi,\sigma,g) \in \mathrm{Aut}(\mathcal{D}) \times \mathrm{Stab}_\mathcal{N}^\dagger(\mathcal{D}) \times \widetilde{GL}^+(2,\mathbb{R})$ is a compatible triple then the (polynomial) categorical entropy and the (polynomial) mass growth of $\Phi$ are linear in $t$ and their slopes are given by the (polynomial) shifting number.

\begin{thm}[see Proposition \ref{linearity}]\label{main-thm-lin}
Suppose $\mathcal{D}$ is saturated and let $(\Phi,\sigma,g) \in \mathrm{Aut}(\mathcal{D}) \times \mathrm{Stab}_\mathcal{N}^\dagger(\mathcal{D}) \times \widetilde{GL}^+(2,\mathbb{R})$ be a compatible triple.
Then we have the following:
\begin{enumerate}
\item $h_t(\Phi) = h_0(\Phi) + \overline{\nu}(\Phi) \cdot t$.
\item $h_{\sigma,t}(\Phi) = h_{\sigma,0}(\Phi) + \overline{\nu}(\Phi) \cdot t$.
\item $h_t^\mathrm{pol}(\Phi) = h_0^\mathrm{pol}(\Phi) + \overline{\nu}^\mathrm{pol}(\Phi) \cdot t$ under the assumption that $h_0(\Phi) = 0$.
\item $h_{\sigma,t}^\mathrm{pol}(\Phi) = h_{\sigma,0}^\mathrm{pol}(\Phi) + \overline{\nu}^\mathrm{pol}(\Phi) \cdot t$ under the assumption that $h_0(\Phi) = 0$.
\end{enumerate}
\end{thm}

As a corollary, we obtain the following necessary and sufficient conditions for the Gromov--Yomdin type theorems for the categorical entropy $h_t(\Phi)$ and the polynomial categorical entropy $h_t^\mathrm{pol}(\Phi)$ (see Definitions \ref{cat-ent} and \ref{pol-cat-ent}).

\begin{thm}[see Corollary \ref{iff-cor}]\label{main-thm-iff}
Suppose $\mathcal{D}$ is saturated and let $(\Phi,\sigma,g) \in \mathrm{Aut}(\mathcal{D}) \times \mathrm{Stab}_\mathcal{N}^\dagger(\mathcal{D}) \times \widetilde{GL}^+(2,\mathbb{R})$ be a compatible triple such that $Z_\sigma$ has spanning image.
Then we have the following:
\begin{enumerate}
\item $h_0(\Phi) = \log\rho(\Phi)$ if and only if $h_t(\Phi) = h_{\sigma,t}(\Phi)$ for all $t \in \mathbb{R}$.
\item $h_0^\mathrm{pol}(\Phi) = s(\Phi)$ if and only if $h_t^\mathrm{pol}(\Phi) = h_{\sigma,t}^\mathrm{pol}(\Phi)$ for all $t \in \mathbb{R}$ under the assumption that $h_0(\Phi) = 0$.
\end{enumerate}
\end{thm}

Let again $(\Phi,\sigma,g) \in \mathrm{Aut}(\mathcal{D}) \times \mathrm{Stab}_\mathcal{N}^\dagger(\mathcal{D}) \times \widetilde{GL}^+(2,\mathbb{R})$ be a compatible triple.
We also study the relationship between the mass growth $h_{\sigma,0}(\Phi)$ and the {\em stable translation length} $\tau^s(\Phi)$ of the action of $\Phi$ on $\mathrm{Stab}_\mathcal{N}^\dagger(\mathcal{D})/\mathbb{C}$ (see Definition \ref{tran-leng}).

\begin{thm}[see Theorem \ref{tran-leng-thm}]
Let $(\Phi,\sigma,g) \in \mathrm{Aut}(\mathcal{D}) \times \mathrm{Stab}_\mathcal{N}^\dagger(\mathcal{D}) \times \widetilde{GL}^+(2,\mathbb{R})$ be a compatible.
Then we have
\begin{equation*}
\tau^s(\Phi) = \log\frac{\rho(M_g)}{\det(M_g)^{1/2}}.
\end{equation*}
In particular, if $Z_\sigma$ has spanning image and $\det(M_g)=1$ then
\begin{equation*}
\tau^s(\Phi) = \log\rho(M_g) = \log\rho(\Phi) = h_{\sigma,0}(\Phi).
\end{equation*}
\end{thm}

For a smooth projective curve $C$, it is known that $h_0(\Phi) = \log\rho(\Phi)$ for any $\Phi \in \mathrm{Aut}(D^b\mathrm{Coh}(C))$ \cite[Theorem 3.1]{Kik1}.
Moreover the space $\mathrm{Stab}_\mathcal{N}(D^b\mathrm{Coh}(C))$ is well-studied \cite{Bri1,Mac,Oka}.
Using this, we can check the technical conditions in Theorems \ref{main-thm-lin} and \ref{main-thm-iff} and obtain the following.

\begin{prop}[see Proposition \ref{curve-prop}]
Let $C$ be a smooth projective curve.
For any $\Phi \in \mathrm{Aut}(D^b\mathrm{Coh}(C))$ and $\sigma \in \mathrm{Stab}_\mathcal{N}(D^b\mathrm{Coh}(C))$, we have
\begin{equation*}
h_t(\Phi) = h_{\sigma,t}(\Phi) = \log\rho(\Phi) + \overline{\nu}(\Phi) \cdot t.
\end{equation*}
Furthermore, if $\Phi \in \mathrm{Aut}_\mathrm{std}(D^b\mathrm{Coh}(C))$ then $h_t(\Phi) = h_{\sigma,t}(\Phi) = \overline{\nu}(\Phi) \cdot t$.
\end{prop}

For another example, let $X$ be an irreducible smooth projective variety of dimension $d \geq 2$.
We consider the triangulated category $D^b\mathrm{Coh}_{(1)}(X)$ defined as the quotient of the bounded derived category $D^b\mathrm{Coh}(X)$ by the full subcategory consisting of complexes whose support has codimension at least 1 \cite[Definition 3.1]{MP} (see Section \ref{MP-example}).
Meinhardt--Partsch \cite{MP} observed that the triangulated category $D^b\mathrm{Coh}_{(1)}(X)$ is closely related to the birational geometry of $X$ and showed that the space of stability conditions on $D^b\mathrm{Coh}_{(1)}(X)$ behaves roughly like that on a curve.
Using their results and Theorem \ref{main-thm-gy}, we can show the following.

\begin{prop}[see Proposition \ref{quot-prop}]
Let $X$ be an irreducible smooth projective variety of dimension $d \geq 2$ and Picard rank 1.
Then, for any $\Phi = f^*(- \otimes \mathcal{L})[m] \in \mathrm{Aut}(D^b\mathrm{Coh}_{(1)}(X))$ with $\mathcal{L}$ numerically trivial and $\sigma \in \mathrm{Stab}_\Gamma(D^b\mathrm{Coh}_{(1)}(X))$, we have
\begin{equation*}
h_{\sigma,0}(\Phi) = \log\rho(\Phi) = \log\rho(f^*|_{N^1(X)}).
\end{equation*}
\end{prop}

\subsection{Conventions}

All categories and functors between them are assumed to be linear over some field.
We also assume that all categories are essentially small and functors between triangulated categories are exact in the sense that they commute with the shift functors (up to functorial isomorphism) and send an exact triangle to an exact triangle.
Also, functors between the derived categories are assumed to be derived.
Whenever we need (especially when we consider the categorical entropy and its relatives), we assume the existence of a split-generator and a stability condition.

We denote the shift functor of a triangulated category by $[1]$ and its $m$-iterate by $[m]$ ($m \in \mathbb{Z}$).
An exact triangle $E \to F \to G \to E[1]$ in a triangulated category will often be depicted as
\begin{equation*}
\begin{tikzcd}
{E} \ar[rr] && {F} \ar[dl]\\
& {G} \ar[ul,dashed] &
\end{tikzcd}.
\end{equation*}

The following is the list of the main notions that we study in this paper:
\begin{center}
\begin{tabular}{|c|c|c|}\hline
$h_t$ & categorical entropy & \multirow{2}{*}{Definition \ref{cat-ent}}\\ \cline{1-2}
$h_\mathrm{cat}$ & categorical entropy at $t=0$ &\\ \hline
$h_t^\mathrm{pol}$ & polynomial categorical entropy & \multirow{2}{*}{Definition \ref{pol-cat-ent}}\\ \cline{1-2}
$h_\mathrm{cat}^\mathrm{pol}$ & polynomial categorical entropy at $t=0$ &\\ \hline
$h_{\sigma,t}$ & mass growth & \multirow{2}{*}{Definition \ref{mass-grow}}\\ \cline{1-2}
$h_\sigma$ & mass growth at $t=0$ &\\ \hline
$h_{\sigma,t}^\mathrm{pol}$ & polynomial mass growth & \multirow{2}{*}{Definition \ref{pol-mass-grow}}\\ \cline{1-2}
$h_\sigma^\mathrm{pol}$ & polynomial mass growth at $t=0$ &\\ \hline
$\overline{\nu}$ & upper shifting number & \multirow{2}{*}{Definition \ref{shift-num}}\\ \cline{1-2}
$\underline{\nu}$ & lower shifting number &\\ \hline
$\overline{\nu}^\mathrm{pol}$ & upper polynomial shifting number & \multirow{2}{*}{Definition \ref{pol-shift-num}}\\ \cline{1-2}
$\underline{\nu}^\mathrm{pol}$ & lower polynomial shifting number &\\ \hline
$\tau$ & translation length & \multirow{2}{*}{Definition \ref{tran-leng}}\\ \cline{1-2}
$\tau^s$ & stable translation length &\\ \hline
$\rho$ & spectral radius & Definition \ref{spec-rad}\\ \hline
$s$ & polynomial growth rate & Definition \ref{pol-grow}\\ \hline
\end{tabular}
\end{center}

\begin{ack}
The authors would like to thank Ed Segal for suggesting that, passed through homological mirror symmetry, the notion of holomorphicity might be formulated in terms of preserving a stability condition.
F.B. was supported by the European Research Council (ERC) under the European Union Horizon 2020 research and innovation programme (grant agreement No.725010).
J.K. was supported by the Institute for Basic Science (IBS-R003-D1).
\end{ack}

\section{Dynamical invariants}

\subsection{Categorical entropy}

Let us first recall the definition of the categorical entropy introduced by Dimitrov--Haiden--Katzarkov--Kontsevich \cite{DHKK}.
It is a categorical analogue of the topological entropy.

\begin{dfn}[{\cite[Definition 2.1]{DHKK}}]
Let $E,F$ be objects of a triangulated category $\mathcal{D}$.
The {\em categorical complexity} of $F$ with respect to $E$ is the function $\delta_t(E,F) : \mathbb{R} \to [0,\infty]$ in $t$ defined by
\begin{equation*}
\delta_t(E,F) \coloneqq \inf
\left\{\sum_{i=1}^k e^{n_i t} \,\left|\,
\begin{tikzcd}[column sep=tiny]
0 \ar[rr] & & * \ar[dl] \ar[rr] & & * \ar[dl] & \cdots & * \ar[rr] & & F \oplus F' \ar[dl]\\
& E[n_1] \ar[ul,dashed] & & E[n_2] \ar[ul,dashed] & & & & E[n_k] \ar[ul,dashed] &
\end{tikzcd}
\right.\right\}
\end{equation*}
if $F \neq 0$, and $\delta_t(E,F) \coloneqq 0$ if $F \cong 0$.
Here the limit runs over all possible cone decompositions of $F \oplus F'$ into $E[n_i]$'s for some $F'$.
\end{dfn}

Recall that an object $G$ of a triangulated category $\mathcal{D}$ is called a {\em split-generator} if the smallest full triangulated subcategory containing $G$ and closed under taking direct summands coincides with $\mathcal{D}$.

\begin{dfn}[{\cite[Definition 2.5]{DHKK}}]\label{cat-ent}
Let $\mathcal{D}$ be a triangulated category and $G$ be a split-generator of $\mathcal{D}$.
The {\em categorical entropy} of an endofunctor $\Phi : \mathcal{D} \to \mathcal{D}$ is the function $h_t(\Phi) : \mathbb{R} \to [-\infty,\infty)$ in $t$ defined by
\begin{equation*}
h_t(\Phi) \coloneqq \lim_{n \to \infty} \frac{\log\delta_t(G,\Phi^nG)}{n}.
\end{equation*}
We denote $h_\mathrm{cat}(\Phi) \coloneqq h_0(\Phi)$.
\end{dfn}

\begin{lem}[{\cite[Lemma 2.6]{DHKK},\cite[Lemma 2.6]{Kik1}}]
The categorical entropy is well-defined in the sense that the limit exists and it does not depend on the choice of a split-generator $G$.
Moreover
\begin{equation*}
h_t(\Phi) = \lim_{n \to \infty} \frac{\log\delta_t(G,\Phi^nG')}{n}
\end{equation*}
for any split-generators $G,G'$.
\end{lem}

Recall that a triangulated category is called {\em saturated} if it admits a smooth and proper dg enhancement.
For example, the bounded derived category of a smooth projective variety is saturated.

For two objects $E,F$ of a saturated triangulated category $\mathcal{D}$, define
\begin{equation*}
\epsilon_t(E,F) \coloneqq \sum_{k \in \mathbb{Z}} \dim \mathrm{Hom}_\mathcal{D}(E,F[k]) e^{-kt}.
\end{equation*}
For endofunctors of a saturated triangulated category, we can use $\epsilon_t$ instead of $\delta_t$ to compute the categorical entropy.

\begin{prop}[{\cite[Theorem 2.7]{DHKK}}]\label{sat-ent}
Let $\mathcal{D}$ be a saturated triangulated category and $\Phi : \mathcal{D} \to \mathcal{D}$ be an endofunctor.
Then
\begin{equation*}
h_t(\Phi) = \lim_{n \to \infty} \frac{\log\epsilon_t(G,\Phi^n G')}{n}
\end{equation*}
for any split-generators $G,G'$.
\end{prop}

\begin{ex}
(1)
Let $\mathcal{D}$ be a saturated triangulated category and $n,m \in \mathbb{Z}$.
Suppose $\Phi : \mathcal{D} \to \mathcal{D}$ satisfies $\Phi^n \cong [m]$.
Then $h_t(\Phi) = \frac{m}{n} t$ \cite[Section 2.6.1]{DHKK}.

(2)
Let $X$ be a smooth projective variety and $\mathcal{L} \in \mathrm{Pic}(X)$.
Then $h_t(- \otimes \mathcal{L}) = 0$ \cite[Lemma 2.14]{DHKK}.

(3)
Let $\mathcal{D}$ be a saturated triangulated category and $E \in \mathcal{D}$ be a {\em $d$-spherical object} where $d \geq 1$.
Denote by $T_E : \mathcal{D} \to \mathcal{D}$ the {\em spherical twist} along $E$.
Then $h_t(T_E) = (1-d)t$ for $t \leq 0$
Furthermore, if $E^\perp \coloneqq \{ F \in \mathcal{D} \,|\, \mathrm{Hom}_\mathcal{D}(E,F[k]) = 0 \text{ for all } k \in \mathbb{Z} \} \neq 0$ then $h_t(T_E) = 0$ for $t \geq 0$ \cite[Theorem 3.1]{Ouc}.
\end{ex}

The categorical entropy measures the exponential growth rate of $\delta_t(G,\Phi^nG)$ with respect to $n$.
In particular, if $\delta_t(G,\Phi^nG)$ grows like a polynomial in $n$, the categorical entropy becomes zero.
Fan--Fu--Ouchi \cite{FFO} introduced the polynomial categorical entropy to measure the polynomial growth rate of $\delta_t(G,\Phi^nG)$ with respect to $n$.

\begin{dfn}[{\cite[Definition 2.4]{FFO}}]\label{pol-cat-ent}
Let $\mathcal{D}$ be a triangulated category and $G$ be a split-generator of $\mathcal{D}$.
For an endofunctor $\Phi : \mathcal{D} \to \mathcal{D}$, define $I^\Phi \coloneqq \{ t \in \mathbb{R} \,|\, h_t(\Phi) \neq -\infty \}$.
The {\em polynomial categorical entropy} of an endofunctor $\Phi : \mathcal{D} \to \mathcal{D}$ is the function $h_t^\mathrm{pol}(\Phi) : I^\Phi \to [-\infty,\infty)$ in $t$ defined by
\begin{equation*}
h_t^\mathrm{pol}(\Phi) \coloneqq \limsup_{n \to \infty} \frac{\log\delta_t(G,\Phi^nG) - nh_t(\Phi)}{\log n}.
\end{equation*}
We denote $h_\mathrm{cat}^\mathrm{pol}(\Phi) \coloneqq h_0^\mathrm{pol}(\Phi)$.
\end{dfn}

Note that $0 \in I_\Phi$ for any $\Phi$ since $h_0(\Phi) \in [0,\infty)$.
Thus we can always consider $h_0^\mathrm{pol}(\Phi)$.

\begin{lem}[{\cite[Lemma 2.6]{FFO}}]
The polynomial categorical entropy does not depend on the choice of a split-generator $G$.
Moreover
\begin{equation*}
h_t^\mathrm{pol}(\Phi) = \limsup_{n \to \infty} \frac{\log\delta_t(G,\Phi^nG') - nh_t(\Phi)}{\log n}.
\end{equation*}
for any split-generators $G,G'$.
\end{lem}

The following is a polynomial analogue of Proposition \ref{sat-ent}.

\begin{prop}[{\cite[Lemma 2.7]{FFO}}]\label{sat-pol-ent}
Let $\mathcal{D}$ be a saturated triangulated category and $\Phi : \mathcal{D} \to \mathcal{D}$ be an endofunctor.
Then
\begin{equation*}
h_t^\mathrm{pol}(\Phi) = \limsup_{n \to \infty} \frac{\log\epsilon_t(G,\Phi^n G') - nh_t(\Phi)}{\log n}
\end{equation*}
for any split-generators $G,G'$.
\end{prop}

\begin{ex}
(1)
Let $\mathcal{D}$ be a saturated triangulated category and $n,m \in \mathbb{Z}$.
Suppose $\Phi : \mathcal{D} \to \mathcal{D}$ satisfies $\Phi^n \cong [m]$.
Then $h_t^\mathrm{pol}(\Phi) = 0$ \cite[Remark 6.3]{FFO}.

(2)
Let $X$ be a smooth projective variety of dimension $d$ and $\mathcal{L} \in \mathrm{Pic}(X)$.
Then $h_t^\mathrm{pol}(- \otimes \mathcal{L})$ is constant in $t$ and $\nu(\mathcal{L}) \leq h_t^\mathrm{pol}(- \otimes \mathcal{L}) \leq d$ where $\nu(\mathcal{L}) \coloneqq \{ n \in \mathbb{Z}_{\geq 0} \,|\, c_1(\mathcal{L})^n \not\equiv 0 \}$ is the {\em numerical dimension} of $\mathcal{L}$ \cite[Proposition 6.4]{FFO}.

(3)
Let $\mathcal{D}$ be a saturated triangulated category and $E \in \mathcal{D}$ be a $d$-spherical object where $d \geq 2$.
Denote by $T_E : \mathcal{D} \to \mathcal{D}$ the spherical twist along $E$.
Then $h_t^\mathrm{pol}(T_E) = 0$ for $t < 0$.
Furthermore, if $E^\perp \neq 0$ then $h_t^\mathrm{pol}(T_E) = 0$ for $t > 0$ \cite[Proposition 6.13]{FFO}.
Note that the value at $t=0$ need not necessarily be zero \cite[Remark 6.16]{FFO}.
\end{ex}

\subsection{Stability conditions}\label{stab-sec}

Let $\mathcal{D}$ be a triangulated category.
The {\em Grothendieck group} $K(\mathcal{D})$ of $\mathcal{D}$ is defined by quotienting the free abelian group generated by the objects of $\mathcal{D}$ by the relations $E-F+G$ whenever there is an exact triangle $E \to F \to G \to E[1]$.
We denote an element of $K(\mathcal{D})$ by $[E]$ for some $E \in \mathcal{D}$.

\begin{dfn}[{\cite[Definition 5.1]{Bri1}}]\label{stab-cond}
Let $\Gamma$ be a lattice of finite rank and $v : K(\mathcal{D}) \to \Gamma$ be a surjective group homomorphism.
A {\em stability condition} $\sigma = (Z_\sigma,\mathcal{P}_\sigma)$ on a triangulated category $\mathcal{D}$ with respect to $(\Gamma,v)$ consists of a group homomorphism $Z_\sigma : \Gamma \to \mathbb{C}$ and a full additive subcategory $\mathcal{P}_\sigma(\phi) \subset \mathcal{D}$ for each $\phi \in \mathbb{R}$ satisfying the following conditions:
\begin{enumerate}
\item If $0 \neq E \in \mathcal{P}_\sigma(\phi)$, then $Z_\sigma(v([E])) = m_\sigma(v([E])) e^{i\pi\phi}$ for some $m_\sigma(v([E])) \in \mathbb{R}_{>0}$.
\item $\mathcal{P}_\sigma(\phi + 1) = \mathcal{P}_\sigma(\phi)[1]$ for every $\phi \in \mathbb{R}$.
\item If $\phi_1 > \phi_2$ and $E_i \in \mathcal{P}_\sigma(\phi_i)$, then $\mathrm{Hom}_\mathcal{D}(E_1,E_2) = 0$.
\item For every $0 \neq E \in \mathcal{D}$, there exist real numbers $\phi_1 > \cdots > \phi_k$ and $E_i \in \mathcal{P}_\sigma(\phi_i)$ which fit into an iterated exact triangle of the form
\begin{equation*}
\begin{tikzcd}[column sep=scriptsize]
{0} \ar[rr] && {*} \ar[rr] \ar[dl] && {*} \ar[dl] & \cdots & {*} \ar[rr] && {E} \ar[dl]\\
& {E_1} \ar[ul,dashed] && {E_2} \ar[ul,dashed] &&&& {E_k} \ar[ul,dashed] &
\end{tikzcd}.
\end{equation*}
\end{enumerate}

For an interval $I \subset \mathbb{R}$, let $\mathcal{P}_\sigma(I)$ be the smallest full extension closed subcategory of $\mathcal{D}$ containing $\mathcal{P}_\sigma(\phi)$ for all $\phi \in I$.
We always assume a stability condition satisfies the following conditions:
\begin{enumerate}
\item[(5)] ({\em Locally finite}) There exists a constant $\eta > 0$ such that the quasi-abelian category $\mathcal{P}_\sigma((\phi-\eta,\phi+\eta))$ is of finite length for each $\phi \in \mathbb{R}$.
\item[(6)] ({\em Support property}) There exists a constant $C > 0$ and a norm $\lVert - \rVert$ on $\Gamma \otimes_\mathbb{Z} \mathbb{R}$ such that
\begin{equation*}
\lVert v([E]) \rVert \leq C \lvert Z_\sigma(v([E])) \rvert
\end{equation*}
for any $E \in \mathcal{P}_\sigma(\phi)$ and $\phi \in \mathbb{R}$.
\end{enumerate}
We denote by $\mathrm{Stab}_\Gamma(\mathcal{D})$ the space of stability conditions (satisfying also the conditions (5) and (6)) on $\mathcal{D}$ with respect to $(\Gamma,v)$.
If $(\Gamma,v)=(K(\mathcal{D}),\mathrm{id})$ then it will be denoted by $\mathrm{Stab}(\mathcal{D})$.
\end{dfn}

\begin{rmk}
Suppose $\mathcal{D}$ is of {\em finite type}, i.e., the graded vector space $\mathrm{Hom}_\mathcal{D}^\bullet(E,F) \coloneqq \bigoplus_{k \in \mathbb{Z}} \mathrm{Hom}_\mathcal{D}(E,F[k])[-k]$ is finite dimensional for any $E,F \in \mathcal{D}$.
Then we can define the {\em Euler form} $\chi : K(\mathcal{D}) \times K(\mathcal{D}) \to \mathbb{Z}$ given by
\begin{equation*}
\chi([E],[F]) \coloneqq \sum_{k \in \mathbb{Z}} (-1)^k \dim\mathrm{Hom}_\mathcal{D}(E,F[k]).
\end{equation*}
The {\em numerical Grothendieck group} $\mathcal{N}(\mathcal{D})$ of $\mathcal{D}$ is defined as
\begin{equation*}
\mathcal{N}(\mathcal{D}) \coloneqq K(\mathcal{D})/\langle [E] \in K(\mathcal{D}) \,|\, \chi([E],-) = 0 \rangle.
\end{equation*}
By an abuse of notation, we denote an element of $\mathcal{N}(\mathcal{D})$ also by $[E]$ for some $E \in \mathcal{D}$.
A triangulated category is called {\em numerically finite} if it is of finite type and the rank of the numerical Grothendieck group is of finite rank.

A stability condition on a numerically finite triangulated category $\mathcal{D}$ with respect to $(\mathcal{N}(\mathcal{D}),\pi)$, where $\pi : K(\mathcal{D}) \to \mathcal{N}(\mathcal{D})$ is the projection, is called a {\em numerical stability condition}.
We denote by $\mathrm{Stab}_\mathcal{N}(\mathcal{D})$ the space of numerical stability conditions on $\mathcal{D}$.
\end{rmk}

From now on, we set $Z_\sigma(E) \coloneqq Z_\sigma(v([E]))$, $m_\sigma(E) \coloneqq m_\sigma(v([E]))$ and so on to simplify the notations.

The group homomorphism $Z_\sigma : \Gamma \to \mathbb{C}$ is called the {\em central charge}.
A non-zero object in $\mathcal{P}_\sigma(\phi)$ is called a {\em $\sigma$-semistable object} of phase $\phi$.
It is known that each $\mathcal{P}_\sigma(\phi)$ is an abelian category \cite[Lemma 5.2]{Bri1}.
We call a simple object of $\mathcal{P}_\sigma(\phi)$ a {\em $\sigma$-stable object} of phase $\phi$.
The {\em heart} of a stability condition $\sigma$ is defined to be $\mathcal{P}_\sigma(0,1]$.
It is the heart of a bounded t-structure determined by $\sigma$, in particular, an abelian category.
In fact, it is known that to give a stability condition $\sigma$ on $\mathcal{D}$ is equivalent to give a bounded t-structure on $\mathcal{D}$ and a {\em stability function} $Z_\sigma : \Gamma \to \mathbb{C}$ on its heart $\mathcal{H}_\sigma$ with a property anologous to Definition \ref{stab-cond} (4) \cite[Proposition 5.3]{Bri1}.
In view of this fact, we sometimes write a stability condition with heart $\mathcal{H}_\sigma$ by $(Z_\sigma,\mathcal{H}_\sigma)$.

The iterated exact triangle as in Definition \ref{stab-cond} (4) is determined uniquely and we call it the {\em Harder--Narasimhan filtration} of $E$ with semistable factors $E_1,\dots,E_k$.
The uniqueness of the Harder--Narasimhan filtration allows us to define $\phi_\sigma^+(E) \coloneqq \phi_1$, $\phi_\sigma^-(E) \coloneqq \phi_k$ and
\begin{equation*}
m_{\sigma,t}(E) \coloneqq \sum_{i=1}^k \lvert Z_\sigma(E_i) \rvert e^{\phi_i t}
\end{equation*}
using the Harder--Narasimhan filtration of $0 \neq E \in \mathcal{D}$.
If $E \cong 0$, we set $m_{\sigma,t}(E) \coloneqq 0$.
We call $m_{\sigma,t}(E)$ the {\em mass} of $E$ and denote $m_\sigma(E) \coloneqq m_{\sigma,0}(E)$.

Bridgeland \cite{Bri1} introduced a generalized metric on $\mathrm{Stab}_\Gamma(\mathcal{D})$ which is defined as
\begin{equation*}
d_B(\sigma,\tau) \coloneqq \sup_{0 \neq E \in \mathcal{D}} \left\{ \lvert \phi_\tau^+(E) - \phi_\sigma^+(E) \rvert,\lvert \phi_\tau^-(E) - \phi_\sigma^-(E) \rvert,\left\lvert\log\frac{m_\tau(E)}{m_\sigma(E)}\right\rvert \right\} \in [0,\infty]
\end{equation*}
for $\sigma,\tau \in \mathrm{Stab}_\Gamma(\mathcal{D})$.
We equip $\mathrm{Stab}_\Gamma(\mathcal{D})$ with the topology induced by this metric.

\begin{thm}[{\cite[Theorem 1.3]{Bri1}}]
The map $\mathrm{Stab}_\Gamma(\mathcal{D}) \to \mathrm{Hom}(\Gamma,\mathbb{C})$ given by $(Z,\mathcal{P}) \mapsto Z$ is a local homeomorphism where $\mathrm{Hom}_\mathbb{Z}(\Gamma,\mathbb{C}) \cong \mathbb{C}^{\mathrm{rk}\,\Gamma}$ is equipped with the natural linear topology.
\end{thm}

Thus the space $\mathrm{Stab}_\Gamma(\mathcal{D})$ of stability conditions carries the structure of a finite dimensional complex manifold induced from the natural complex structure on $\mathrm{Hom}_\mathbb{Z}(\Gamma,\mathbb{C}) \cong \mathbb{C}^{\mathrm{rk}\,\Gamma}$.

Let $\mathrm{Aut}(\mathcal{D})$ be the group of autoequivalences of $\mathcal{D}$ and assume that there is a group homomorphism $[-]_\Gamma : \mathrm{Aut}(\mathcal{D}) \to \mathrm{Aut}_\mathbb{Z}(\Gamma)$ (or $G \to \mathrm{Aut}_\mathbb{Z}(\Gamma)$ for a subgroup $G$ of $\mathrm{Aut}(\mathcal{D})$) which makes the diagram
\begin{equation*}
\begin{tikzcd}
K(\mathcal{D}) \ar[r,"{[\Phi]}"] \ar[d,"v",swap] & K(\mathcal{D}) \ar[d,"v"]\\
\Gamma \ar[r,"{[\Phi]}_\Gamma",swap] & \Gamma
\end{tikzcd}
\end{equation*}
commute, where $[\Phi] : K(\mathcal{D}) \to K(\mathcal{D})$ denotes the induced group homomorphism.
The group $\mathrm{Aut}(\mathcal{D})$ acts on $\mathrm{Stab}_\Gamma(\mathcal{D})$ from the left.
For $\Phi \in \mathrm{Aut}(\mathcal{D})$, the action is given by
\begin{equation*}
\Phi \cdot (Z,\mathcal{P}) \coloneqq (Z',\mathcal{P}')
\end{equation*}
where $Z' \coloneqq Z \circ [\Phi]_\Gamma^{-1}$ and $\mathcal{P}'(\phi) \coloneqq \Phi(\mathcal{P}(\phi))$.

\begin{rmk}
Alternatively, we can consider the group $\mathrm{Aut}_\Gamma(\mathcal{D}) \coloneqq \{(\Phi,a) \in \mathrm{Aut}(\mathcal{D}) \times \mathrm{Aut}_\mathbb{Z}(\Gamma) \,|\, v \circ [\Phi] = a \circ v\}$ as in \cite[Section 6.3]{KP}.
Then all our statements can be generalized by replacing $[\Phi]_\Gamma$ by $a$.
However, for the sake of simplicity, we assume throughout the paper that every $\Phi \in \mathrm{Aut}(\mathcal{D})$ admits $[\Phi]_\Gamma \in \mathrm{Aut}_\mathbb{Z}(\Gamma)$ such that $(\Phi,[\Phi]_\Gamma) \in \mathrm{Aut}_\Gamma(\mathcal{D})$ and the map $\mathrm{Aut}(\mathcal{D}) \ni \Phi \mapsto [\Phi]_\Gamma \in \mathrm{Aut}_\mathbb{Z}(\Gamma)$ defines a group homomorphism.
\end{rmk}

Recall that the universal cover $\widetilde{GL}^+(2,\mathbb{R})$ of $GL^+(2,\mathbb{R})$ consists of elements $(M,f)$ where $M \in GL^+(2,\mathbb{R})$ and $f : \mathbb{R} \to \mathbb{R}$ is an increasing function satisfying $f(\phi+1)=f(\phi)+1$ and $Me^{i\pi\phi} \in \mathbb{R}_{>0}e^{i\pi f(\phi)}$.
The group $\widetilde{GL}^+(2,\mathbb{R})$ acts on $\mathrm{Stab}_\Gamma(\mathcal{D})$ from the right.
For $(M,f) \in \widetilde{GL}^+(2,\mathbb{R})$, the action is given by
\begin{equation*}
(Z,\mathcal{P}) \cdot (M,f) \coloneqq (Z'',\mathcal{P}'')
\end{equation*}
where $Z'' \coloneqq M^{-1} \cdot Z$ and $\mathcal{P}''(\phi) \coloneqq \mathcal{P}(f(\phi))$.

Moreover $\mathbb{C}$ also acts on $\mathrm{Stab}_\Gamma(\mathcal{D})$ from the right through the map
\begin{equation*}
\mathbb{C} \ni \alpha \mapsto \left( M_\alpha \coloneqq e^{-\pi\Im\alpha} \begin{pmatrix} \cos\pi\Re\alpha & -\sin\pi\Re\alpha\\ \sin\pi\Re\alpha & \cos\pi\Re\alpha \end{pmatrix}, f_\alpha : \phi \mapsto \phi+\Re\alpha \right) \in \widetilde{GL}^+(2,\mathbb{R})
\end{equation*}
where $\Re\alpha$ (resp. $\Im\alpha$) is the real (resp. imaginary) part of $\alpha$.
More explicitly, $\alpha \in \mathbb{C}$ acts as
\begin{equation*}
(Z,\mathcal{P}) \cdot \alpha \coloneqq (Z''',\mathcal{P}''')
\end{equation*}
where $Z''' \coloneqq e^{-i\pi\alpha} Z$ and $\mathcal{P}'''(\phi) \coloneqq \mathcal{P}(\phi+\Re\alpha)$.

In the rest of this paper, we will often fix a connected component of $\mathrm{Stab}_\Gamma(\mathcal{D})$ and denote it by $\mathrm{Stab}_\Gamma^\dagger(\mathcal{D})$.
The generalized metric $d_B$ becomes a genuine metric on $\mathrm{Stab}_\Gamma^\dagger(\mathcal{D})$, i.e., it takes  values on $[0,\infty)$.
Also note that the action of $\widetilde{GL}^+(2,\mathbb{R})$ preserves the connected component whereas that of $\mathrm{Aut}(\mathcal{D})$ need not.

\subsection{Mass growth}
The mass growth was introduced by Ikeda \cite{Ike} as a categorical analogue of the volume growth.

\begin{dfn}[{\cite[Definition 3.2]{Ike}}]\label{mass-grow}
Let $\sigma \in \mathrm{Stab}_\Gamma(\mathcal{D})$.
The {\em mass growth} with respect to an endofunctor $\Phi : \mathcal{D} \to \mathcal{D}$ is the function $h_{\sigma,t}(\Phi) : \mathbb{R} \to [-\infty,\infty)$ in $t$ defined by
\begin{equation*}
h_{\sigma,t}(\Phi) \coloneqq \sup_{0 \neq E \in \mathcal{D}} \left( \limsup_{n \to \infty} \frac{\log m_{\sigma,t}(\Phi^nE)}{n} \right).
\end{equation*}
We denote $h_\sigma(\Phi) \coloneqq h_{\sigma,0}(\Phi)$.
\end{dfn}

\begin{lem}[{\cite[Proposition 3.10]{Ike}}]\label{mass-grow-lem}
If $\sigma,\tau \in \mathrm{Stab}_\Gamma(\mathcal{D})$ are in the same connected component, then $h_{\sigma,t}(\Phi) = h_{\tau,t}(\Phi)$.
In particular, the mass growth only depends on the connected component in which $\sigma \in \mathrm{Stab}_\Gamma(\mathcal{D})$ lies.
\end{lem}

A stability condition $\sigma \in \mathrm{Stab}_\Gamma(\mathcal{D})$ is called an {\em algebraic stability condition} if its heart is a finite length abelian category with finitely many simple objects.

\begin{thm}[{\cite[Theorems 3.5 and 3.14]{Ike}}]\label{mass-grow-thm}
Let $\sigma \in \mathrm{Stab}_\Gamma^\dagger(\mathcal{D})$ and $\Phi : \mathcal{D} \to \mathcal{D}$ be an endofunctor.
\begin{enumerate}
\item For any split-generator $G \in \mathcal{D}$, we have
\begin{equation*}
h_{\sigma,t}(\Phi) = \limsup_{n \to \infty} \frac{\log m_{\sigma,t}(\Phi^nG)}{n}.
\end{equation*}
\item We have an inequality
\begin{equation*}
h_{\sigma,t}(\Phi) \leq h_t(\Phi)
\end{equation*}
where the equality holds if $\sigma$ lies in a connected component containing an algebraic stability condition.
\end{enumerate}
\end{thm}

A polynomial analogue of the mass growth was introduced by Fan--Fu--Ouchi \cite{FFO}.

\begin{dfn}[{\cite[Definition 3.3]{FFO}}]\label{pol-mass-grow}
Let $\sigma \in \mathrm{Stab}_\Gamma(\mathcal{D})$.
For an endofunctor $\Phi : \mathcal{D} \to \mathcal{D}$, define $I_\sigma^\Phi \coloneqq \{ t \in \mathbb{R} \,|\, h_{\sigma,t}(\Phi) \neq -\infty \}$.
The {\em polynomial mass growth} with respect to an endofunctor $\Phi : \mathcal{D} \to \mathcal{D}$ is the function $h_t^\mathrm{pol}(\Phi) : I_\sigma^\Phi \to [-\infty,\infty)$ in $t$ defined by
\begin{equation*}
h_{\sigma,t}(\Phi) \coloneqq \sup_{0 \neq E \in \mathcal{D}} \left( \limsup_{n \to \infty} \frac{\log m_{\sigma,t}(\Phi^nE) - nh_{\sigma,t}(\Phi)}{\log n} \right).
\end{equation*}
We denote $h_\sigma^\mathrm{pol}(\Phi) \coloneqq h_{\sigma,0}^\mathrm{pol}(\Phi)$.
\end{dfn}

Note that $0 \in I_\Phi$ for any $\Phi$ since $h_{\sigma,0}(\Phi) \in [0,\infty)$.
Thus we can always consider $h_{\sigma,0}^\mathrm{pol}(\Phi)$.

\begin{lem}[{\cite[Lemma 3.5]{FFO}}]\label{pol-mass-grow-lem}
If $\sigma,\tau \in \mathrm{Stab}_\Gamma(\mathcal{D})$ are in the same connected component, then $h_{\sigma,t}^\mathrm{pol}(\Phi) = h_{\tau,t}^\mathrm{pol}(\Phi)$.
In particular, the polynomial mass growth only depends on the connected component in which $\sigma \in \mathrm{Stab}_\Gamma(\mathcal{D})$ lies.
\end{lem}

The following is a polynomial analogue of Theorem \ref{mass-grow-thm}.

\begin{thm}[{\cite[Lemmas 3.4, 3.6 and 3.7]{FFO}}]\label{pol-mass-grow-thm}
Let $\sigma \in \mathrm{Stab}_\Gamma^\dagger(\mathcal{D})$ and $\Phi : \mathcal{D} \to \mathcal{D}$ be an endofunctor.
\begin{enumerate}
\item For any split-generator $G \in \mathcal{D}$, we have
\begin{equation*}
h_{\sigma,t}^\mathrm{pol}(\Phi) = \limsup_{n \to \infty} \frac{\log m_{\sigma,t}(\Phi^nG) - nh_{\sigma,t}(\Phi)}{\log n}.
\end{equation*}
\item We have an inequality
\begin{equation*}
h_{\sigma,t}^\mathrm{pol}(\Phi) \leq h_t^\mathrm{pol}(\Phi)
\end{equation*}
where the equality holds if $\sigma$ lies in a connected component containing an algebraic stability condition.
\end{enumerate}
\end{thm}

\subsection{Shifting number}

The shifting number was introduced by Fan--Filip \cite{FF} as a categorical analogue of the Poincar\'e translation number.
It also can be considered as a generalization of the Serre dimension introduced by Elagin--Lunts \cite{EL}.

In this section, we only consider saturated triangulated categories.
Note that, for an endofunctor $\Phi$ of a saturated triangulated category, its categorical entropy $h_t(\Phi)$ is a real-valued convex function in $t$ \cite[Theorem 2.1.6]{FF}.

\begin{dfn}[{\cite[Definition 2.1.8]{FF}}]\label{shift-num}
Let $\mathcal{D}$ be a saturated triangulated category.
The {\em upper shifting number} (resp. {\em lower shifting number}) of an endofunctor $\Phi : \mathcal{D} \to \mathcal{D}$ is the number $\overline{\nu}(\Phi) \in \mathbb{R}$ (resp. $\underline{\nu}(\Phi) \in \mathbb{R}$) defined by
\begin{equation*}
\overline{\nu}(\Phi) \coloneqq \lim_{t \to \infty} \frac{h_t(\Phi)}{t} \quad \left( \text{resp. } \underline{\nu}(\Phi) \coloneqq \lim_{t \to -\infty} \frac{h_t(\Phi)}{t} \right).
\end{equation*}
\end{dfn}

\begin{rmk}
Every saturated triangulated category has the Serre functor \cite[Corollary 3.5]{BK2}.
If $S$ denotes the Serre functor of a saturated triangulated category $\mathcal{D}$, its upper shifting number $\overline{\nu}(S)$ (resp. lower shifting number $\underline{\nu}(S)$) is called the {\em upper Serre dimension} (resp. {\em lower Serre dimension}) of $\mathcal{D}$ \cite[Definition 5.3]{EL}.
\end{rmk}

Now, for two objects $E,F$ of a saturated triangulated category $\mathcal{D}$, define
\begin{align*}
\epsilon^+(E,F) &\coloneqq \max\{ k \in \mathbb{Z} \,|\, \mathrm{Hom}_\mathcal{D}(E,F[-k]) \neq 0 \},\\
\epsilon^-(E,F) &\coloneqq \min\{ k \in \mathbb{Z} \,|\, \mathrm{Hom}_\mathcal{D}(E,F[-k]) \neq 0 \}.
\end{align*}

\begin{prop}[{\cite[Theorem 2.1.7]{FF}}]\label{shift-ent}
The upper (resp. lower) shifting number is well-defined in the sense that the limit exists.
Moreover
\begin{equation*}
\overline{\nu}(\Phi) = \lim_{n \to \infty} \frac{\epsilon^+(G,\Phi^nG')}{n} \quad \left( \text{resp. } \underline{\nu}(\Phi) = \lim_{n \to \infty} \frac{\epsilon^-(G,\Phi^nG')}{n} \right)
\end{equation*}
for any split-generators $G,G'$.
In addition, we have
\begin{gather*}
\overline{\nu}(\Phi) \cdot t \leq h_t(\Phi) \leq h_\mathrm{cat}(\Phi) + \overline{\nu}(\Phi) \cdot t \quad (t \geq 0),\\
\underline{\nu}(\Phi) \cdot t \leq h_t(\Phi) \leq h_\mathrm{cat}(\Phi) + \underline{\nu}(\Phi) \cdot t \quad (t \leq 0).
\end{gather*}
\end{prop}

There is another lower bound for the categorical entropy.

\begin{prop}[{\cite[Theorem 6.14]{EL}}]\label{shift-ent-lower}
Let $\mathcal{D}$ be a saturated triangulated category and $\Phi : \mathcal{D} \to \mathcal{D}$ be an endofunctor.
Then we have
\begin{gather*}
h_t(\Phi) \geq h_\mathrm{cat}(\Phi) + \underline{\nu}(\Phi) \cdot t \quad (t \geq 0),\\
h_t(\Phi) \geq h_\mathrm{cat}(\Phi) + \overline{\nu}(\Phi) \cdot t \quad (t \leq 0).
\end{gather*}
\end{prop}

The shifting number can be described in terms of stability conditions.

\begin{prop}[{\cite[Theorem 2.2.6]{FF}}]\label{shift-stab}
Let $\mathcal{D}$ be a saturated triangulated category and $\sigma \in \mathrm{Stab}_\Gamma^\dagger(\mathcal{D})$.
For any endofunctor $\Phi : \mathcal{D} \to \mathcal{D}$ and a split-generator $G$, we have
\begin{align*}
\overline{\nu}(\Phi) &= \lim_{t \to \infty} \frac{h_{\sigma,t}(\Phi)}{t} = \lim_{n \to \infty} \frac{\phi_\sigma^+(\Phi^nG)}{n},\\
\underline{\nu}(\Phi) &= \lim_{t \to -\infty} \frac{h_{\sigma,t}(\Phi)}{t} = \lim_{n \to \infty} \frac{\phi_\sigma^-(\Phi^nG)}{n}.
\end{align*}
In addition, we have
\begin{gather*}
\overline{\nu}(\Phi) \cdot t \leq h_{\sigma,t}(\Phi) \leq h_\sigma(\Phi) + \overline{\nu}(\Phi) \cdot t \quad (t \geq 0),\\
\underline{\nu}(\Phi) \cdot t \leq h_{\sigma,t}(\Phi) \leq h_\sigma(\Phi) + \underline{\nu}(\Phi) \cdot t \quad (t \leq 0).
\end{gather*}
\end{prop}

This also shows that the asymptotic behavior of the mass growth $h_{\sigma,t}(\Phi)$ does not depend on the choice of $\sigma \in \mathrm{Stab}_\Gamma(\mathcal{D})$.

A lower bound analogous to Proposition \ref{shift-ent-lower} can be proved for the mass growth.

\begin{prop}\label{shift-stab-lower}
Let $\mathcal{D}$ be a saturated triangulated category and $\sigma \in \mathrm{Stab}_\Gamma^\dagger(\mathcal{D})$.
For any endofunctor $\Phi : \mathcal{D} \to \mathcal{D}$, we have
\begin{gather*}
h_{\sigma,t}(\Phi) \geq h_\sigma(\Phi) + \underline{\nu}(\Phi) \cdot t \quad (t \geq 0),\\
h_{\sigma,t}(\Phi) \geq h_\sigma(\Phi) + \overline{\nu}(\Phi) \cdot t \quad (t \leq 0).
\end{gather*}
\end{prop}

\begin{rmk}
While we were writing the paper, Woolf \cite{Woo2} uploaded a paper containing a result similar to Proposition \ref{shift-stab-lower} (see \cite[Proposition 3.2]{Woo2}).
\end{rmk}

\begin{proof}
Fix a split-generator $G$.
For any $t \geq 0$, we have
\begin{equation*}
m_{\sigma,t}(\Phi^nG) \geq m_\sigma(\Phi^nG) e^{\phi_\sigma^-(\Phi^nG) \cdot t}.
\end{equation*}
Then, applying Theorem \ref{mass-grow-thm} (1) and Proposition \ref{shift-stab}, we obtain the desired lower bound.

A similar argument can be applied for $t \leq 0$.
\end{proof}

Let us introduce a polynomial analogue of the shifting number for an endofunctor with vanishing categorical entropy.

\begin{dfn}\label{pol-shift-num}
Let $\mathcal{D}$ be a saturated triangulated category and $\Phi : \mathcal{D} \to \mathcal{D}$ be an endofunctor with $h_\mathrm{cat}(\Phi) = 0$.
The {\em upper polynomial shifting number} (resp. {\em lower polynomial shifting number}) of $\Phi$ is the number $\overline{\nu}^\mathrm{pol}(\Phi) \in [-\infty,\infty)$ (resp. $\underline{\nu}^\mathrm{pol}(\Phi) \in [-\infty,\infty)$) defined by
\begin{equation*}
\overline{\nu}^\mathrm{pol}(\Phi) \coloneqq \lim_{t \to \infty} \frac{h_t^\mathrm{pol}(\Phi)}{t} \quad \left( \text{resp. } \underline{\nu}^\mathrm{pol}(\Phi) \coloneqq \lim_{t \to -\infty} \frac{h_t^\mathrm{pol}(\Phi)}{t} \right).
\end{equation*}
\end{dfn}

The following is a polynomial analogue of Proposition \ref{shift-ent}.

\begin{prop}\label{pol-shift-ent}
The upper (resp. lower) dimension is well-defined in the sense that the limit exists.
Moreover
\begin{gather*}
\overline{\nu}^\mathrm{pol}(\Phi) = \limsup_{n \to \infty} \frac{\epsilon^+(G,\Phi^nG') - n\overline{\nu}(\Phi)}{\log n}\\
\left( \text{resp. } \underline{\nu}^\mathrm{pol}(\Phi) = \liminf_{n \to \infty} \frac{\epsilon^-(G,\Phi^nG') - n\underline{\nu}(\Phi)}{\log n} \right)
\end{gather*}
for any split-generators $G,G'$.
In addition, we have
\begin{gather*}
\overline{\nu}^\mathrm{pol}(\Phi) \cdot t \leq h_t^\mathrm{pol}(\Phi) \leq h_\mathrm{cat}^\mathrm{pol}(\Phi) + \overline{\nu}^\mathrm{pol}(\Phi) \cdot t \quad (t \geq 0),\\
\underline{\nu}^\mathrm{pol}(\Phi) \cdot t \leq h_t^\mathrm{pol}(\Phi) \leq h_\mathrm{cat}^\mathrm{pol}(\Phi) + \underline{\nu}^\mathrm{pol}(\Phi) \cdot t \quad (t \leq 0).
\end{gather*}
\end{prop}

\begin{proof}
For $t \geq 0$, we have
\begin{equation*}
e^{\epsilon^+(G,\Phi^nG') \cdot t} \leq \epsilon_t(G,\Phi^nG') \leq \epsilon_0(G,\Phi^nG') e^{\epsilon^+(G,\Phi^nG') \cdot t}
\end{equation*}
and therefore
\begin{align*}
\left( \limsup_{n \to \infty} \frac{\epsilon^+(G,\Phi^nG') - n\overline{\nu}(\Phi)}{\log n} \right) t
&\leq \limsup_{n \to \infty} \frac{\log \epsilon_t(G,\Phi^nG') - n\overline{\nu}(\Phi) \cdot t}{\log n}\\
&\leq \limsup_{n \to \infty} \left( \frac{\log \epsilon_0(G,\Phi^nG')}{\log n} + \frac{\epsilon^+(G,\Phi^nG')  - n\overline{\nu}(\Phi)}{\log n} \cdot t \right)\\
&= h_\mathrm{cat}^\mathrm{pol}(\Phi) + \left( \limsup_{n \to \infty} \frac{\epsilon^+(G,\Phi^nG') - n\overline{\nu}(\Phi)}{\log n} \right) t.
\end{align*}
Since $h_\mathrm{cat}(\Phi) = 0$ by assumption, we see that
\begin{equation*}
\limsup_{n \to \infty} \frac{\log \epsilon_t(G,\Phi^nG') - n\overline{\nu}(\Phi) \cdot t}{\log n} = \limsup_{n \to \infty} \frac{\log \epsilon_t(G,\Phi^nG') - nh_t(\Phi)}{\log n} = h_t^\mathrm{pol}(\Phi)
\end{equation*}
by Propositions \ref{sat-pol-ent} and \ref{shift-ent}.
Thus it follows that
\begin{equation*}
\overline{\nu}^\mathrm{pol}(\Phi) = \lim_{t \to \infty} \frac{h_t^\mathrm{pol}(\Phi)}{t} = \limsup_{n \to \infty} \frac{\epsilon^+(G,\Phi^nG') - n\overline{\nu}(\Phi)}{\log n}
\end{equation*}
and
\begin{equation*}
\overline{\nu}^\mathrm{pol}(\Phi) \cdot t \leq h_t^\mathrm{pol}(\Phi) \leq h_\mathrm{cat}^\mathrm{pol}(\Phi) + \overline{\nu}^\mathrm{pol}(\Phi) \cdot t
\end{equation*}
for every $t \geq 0$.

The case $t \leq 0$ can be similarly proved.
\end{proof}

The following is a polynomial analogue of Proposition \ref{shift-ent-lower}.

\begin{prop}\label{pol-shift-ent-lower}
Let $\mathcal{D}$ be a saturated triangulated category and $\Phi : \mathcal{D} \to \mathcal{D}$ be an endofunctor with $h_\mathrm{cat}(\Phi)=0$.
Then we have
\begin{gather*}
h_t^\mathrm{pol}(\Phi) \geq h_\mathrm{cat}^\mathrm{pol}(\Phi) + \underline{\nu}^\mathrm{pol}(\Phi) \cdot t \quad (t \geq 0),\\
h_t^\mathrm{pol}(\Phi) \geq h_\mathrm{cat}^\mathrm{pol}(\Phi) + \overline{\nu}^\mathrm{pol}(\Phi) \cdot t \quad (t \leq 0).
\end{gather*}
\end{prop}

\begin{proof}
Fix split-generators $G,G'$.
For any $t \geq 0$, we have
\begin{equation*}
\epsilon_t(G,\Phi^nG') \geq \epsilon_0(G,\Phi^nG') e^{\epsilon^-(G,\Phi^nG') \cdot t}.
\end{equation*}
As in the proof of Proposition \ref{pol-shift-ent}, we get
\begin{align*}
h_t^\mathrm{pol}(\Phi)
&\geq h_\mathrm{cat}^\mathrm{pol}(\Phi) + \limsup_{n \to \infty} \left( \frac{\epsilon^-(G,\Phi^nG') - n\underline{\nu}(\Phi)}{\log n} \cdot t \right)\\
&\geq h_\mathrm{cat}^\mathrm{pol}(\Phi) + \liminf_{n \to \infty} \left( \frac{\epsilon^-(G,\Phi^nG') - n\underline{\nu}(\Phi)}{\log n} \cdot t \right)\\
&= h_\mathrm{cat}^\mathrm{pol}(\Phi) + \left( \liminf_{n \to \infty} \frac{\epsilon^-(G,\Phi^nG') - n\underline{\nu}(\Phi)}{\log n} \right) t\\
&= h_\mathrm{cat}^\mathrm{pol}(\Phi) + \underline{\nu}^\mathrm{pol}(\Phi) \cdot t.
\end{align*}

For $t \leq 0$, we similarly obtain
\begin{align*}
h_t^\mathrm{pol}(\Phi)
&\geq h_\mathrm{cat}^\mathrm{pol}(\Phi) + \limsup_{n \to \infty} \left( \frac{\epsilon^+(G,\Phi^nG') - n\overline{\nu}(\Phi)}{\log n} \cdot t \right)\\
&\geq h_\mathrm{cat}^\mathrm{pol}(\Phi) + \liminf_{n \to \infty} \left( \frac{\epsilon^+(G,\Phi^nG') - n\overline{\nu}(\Phi)}{\log n} \cdot t \right)\\
&= h_\mathrm{cat}^\mathrm{pol}(\Phi) + \left( \limsup_{n \to \infty} \frac{\epsilon^+(G,\Phi^nG') - n\overline{\nu}(\Phi)}{\log n} \right) t\\
&= h_\mathrm{cat}^\mathrm{pol}(\Phi) + \overline{\nu}^\mathrm{pol}(\Phi) \cdot t
\end{align*}
as desired.
\end{proof}

The following is a polynomial analogue of Proposition \ref{shift-stab}.

\begin{prop}\label{pol-shift-stab}
Let $\mathcal{D}$ be a saturated triangulated category and $\sigma \in \mathrm{Stab}_\Gamma^\dagger(\mathcal{D})$.
For any endofunctor $\Phi : \mathcal{D} \to \mathcal{D}$ with $h_\mathrm{cat}(\Phi)=0$ and a split-generator $G$, we have
\begin{align*}
\overline{\nu}^\mathrm{pol}(\Phi) &= \lim_{t \to \infty} \frac{h_{\sigma,t}^\mathrm{pol}(\Phi)}{t} = \limsup_{n \to \infty} \frac{\phi_\sigma^+(\Phi^nG) - n\overline{\nu}(\Phi)}{\log n},\\
\underline{\nu}^\mathrm{pol}(\Phi) &= \lim_{t \to -\infty} \frac{h_{\sigma,t}^\mathrm{pol}(\Phi)}{t} = \liminf_{n \to \infty} \frac{\phi_\sigma^-(\Phi^nG) - n\underline{\nu}(\Phi)}{\log n}.
\end{align*}
In addition, we have
\begin{gather*}
\overline{\nu}^\mathrm{pol}(\Phi) \cdot t \leq h_{\sigma,t}^\mathrm{pol}(\Phi) \leq h_\sigma^\mathrm{pol}(\Phi) + \overline{\nu}^\mathrm{pol}(\Phi) \cdot t \quad (t \geq 0),\\
\underline{\nu}^\mathrm{pol}(\Phi) \cdot t \leq h_{\sigma,t}^\mathrm{pol}(\Phi) \leq h_\sigma^\mathrm{pol}(\Phi) + \underline{\nu}^\mathrm{pol}(\Phi) \cdot t \quad (t \leq 0).
\end{gather*}
\end{prop}

\begin{proof}
Note that, by Definition \ref{stab-cond} (6), we can find a constant $C>0$ such that $\lvert Z_\sigma(E) \rvert > C$ for every $\sigma$-semistable object $E$.

For $t \geq 0$, we thus have
\begin{equation*}
Ce^{\phi_\sigma^+(\Phi^nG) \cdot t} \leq m_{\sigma,t}(\Phi^nG) \leq m_{\sigma,0}(\Phi^nG)e^{\phi_\sigma^+(\Phi^nG) \cdot t}
\end{equation*}
and therefore
\begin{align}\label{ineq-want-pos}
\left( \limsup_{n \to \infty} \frac{\phi_\sigma^+(\Phi^nG) - n\overline{\nu}(\Phi)}{\log n} \right) t
&\leq \limsup_{n \to \infty} \frac{\log m_{\sigma,t}(\Phi^nG) - n\overline{\nu}(\Phi) \cdot t}{\log n}\nonumber\\
&\leq \limsup_{n \to \infty} \left( \frac{\log m_{\sigma,0}(\Phi^nG)}{\log n} + \frac{\phi_\sigma^+(\Phi^nG)  - n\overline{\nu}(\Phi)}{\log n} \cdot t \right)\nonumber\\
&= h_\sigma^\mathrm{pol}(\Phi) + \left( \limsup_{n \to \infty} \frac{\phi_\sigma^+(\Phi^nG) - n\overline{\nu}(\Phi)}{\log n} \right) t.
\end{align}
Since $h_\mathrm{cat}(\Phi) = 0$ by assumption, we see that $h_\sigma(\Phi) = 0$ by Theorem \ref{mass-grow-thm} (2).
Thus we get
\begin{equation}\label{eq-ent}
\limsup_{n \to \infty} \frac{\log m_{\sigma,t}(\Phi^nG) - n\overline{\nu}(\Phi) \cdot t}{\log n} = \limsup_{n \to \infty} \frac{\log m_{\sigma,t}(\Phi^nG) - nh_{\sigma,t}(\Phi)}{\log n} = h_{\sigma,t}^\mathrm{pol}(\Phi)
\end{equation}
by Theorem \ref{pol-mass-grow-thm} (1) and Proposition \ref{shift-stab}.
Hence it follows that
\begin{equation*}
\lim_{t \to \infty} \frac{h_{\sigma,t}^\mathrm{pol}(\Phi)}{t} = \limsup_{n \to \infty} \frac{\phi_\sigma^+(\Phi^nG) - n\overline{\nu}(\Phi)}{\log n}.
\end{equation*}

On the other hand, note that, for $E,F \in \mathcal{D}$, the condition $\mathrm{Hom}_\mathcal{D}(E,F[-k]) \neq 0$ implies that $\phi_\sigma^+(F) - \phi_\sigma^-(E) \geq k$.
In particular, we have $\phi_\sigma^+(F) - \phi_\sigma^-(E) \geq \epsilon^+(E,F)$.
Using this, we get
\begin{equation}\label{ineq1}
\phi_\sigma^+(\Phi^nG) - \phi_\sigma^-(G) \geq \epsilon^+(G,\Phi^nG)
\end{equation}
and hence
\begin{align*}
\lim_{t \to \infty} \frac{h_{\sigma,t}^\mathrm{pol}(\Phi)}{t}
&= \limsup_{n \to \infty} \frac{\phi_\sigma^+(\Phi^nG) - n\overline{\nu}(\Phi)}{\log n}\\
&\geq \limsup_{n \to \infty} \frac{\epsilon^+(G,\Phi^nG) - n\overline{\nu}(\Phi)}{\log n}\\
&= \overline{\nu}^\mathrm{pol}(\Phi)\\
&= \lim_{t \to \infty} \frac{h_t^\mathrm{pol}(\Phi)}{t}\\
&\geq \lim_{t \to \infty} \frac{h_{\sigma,t}^\mathrm{pol}(\Phi)}{t}.
\end{align*}
Therefore we obtain
\begin{equation}\label{eq-upper-shift}
\overline{\nu}^\mathrm{pol}(\Phi) = \limsup_{n \to \infty} \frac{\phi_\sigma^+(\Phi^nG) - n\overline{\nu}(\Phi)}{\log n}.
\end{equation}
Then, substituting the equalities \eqref{eq-ent} and \eqref{eq-upper-shift} into the inequality \eqref{ineq-want-pos}, we obtain the desired inequality.

Next, let us consider the case $t \leq 0$.
As before, we obtain
\begin{equation}\label{ineq-want-neg}
\left( \liminf_{n \to \infty} \frac{\phi_\sigma^-(\Phi^nG) - n\underline{\nu}(\Phi)}{\log n} \right) t \leq h_{\sigma,t}^\mathrm{pol}(\Phi) \leq h_\sigma^\mathrm{pol}(\Phi) + \left( \liminf_{n \to \infty} \frac{\phi_\sigma^-(\Phi^nG) - n\underline{\nu}(\Phi)}{\log n} \right) t
\end{equation}
and therefore
\begin{equation*}
\lim_{t \to -\infty} \frac{h_{\sigma,t}^\mathrm{pol}(\Phi)}{t} = \liminf_{n \to \infty} \frac{\phi_\sigma^-(\Phi^nG) - n\underline{\nu}(\Phi)}{\log n}.
\end{equation*}
Let $S$ be the Serre functor of $\mathcal{D}$.	
If $0 \neq \mathrm{Hom}_\mathcal{D}(E,F[-k]) \cong \mathrm{Hom}_\mathcal{D}(F[-k],SE)$, then $\phi_\sigma^+(SE) - \phi_\sigma^-(F) \geq -k$.
In particular, we have $\phi_\sigma^+(SE) - \phi_\sigma^-(F) \geq -\epsilon^-(E,F)$.
Using this, we get
\begin{equation}\label{ineq2}
\epsilon^-(G,\Phi^nG) \geq \phi_\sigma^-(\Phi^nG) - \phi_\sigma^+(SG)
\end{equation}
and hence
\begin{align*}
\lim_{t \to -\infty} \frac{h_{\sigma,t}^\mathrm{pol}(\Phi)}{t}
&= \liminf_{n \to \infty} \frac{\phi_\sigma^-(\Phi^nG) - n\underline{\nu}(\Phi)}{\log n}\\
&\leq \liminf_{n \to \infty} \frac{\epsilon^-(G,\Phi^nG) - n\underline{\nu}(\Phi)}{\log n}\\
&= \underline{\nu}^\mathrm{pol}(\Phi)\\
&= \lim_{t \to -\infty} \frac{h_t^\mathrm{pol}(\Phi)}{t}\\
&\leq \lim_{t \to -\infty} \frac{h_{\sigma,t}^\mathrm{pol}(\Phi)}{t}.
\end{align*}
Therefore we obtain
\begin{equation}\label{eq-lower-shift}
\underline{\nu}^\mathrm{pol}(\Phi) = \liminf_{n \to \infty} \frac{\phi_\sigma^-(\Phi^nG) - n\underline{\nu}(\Phi)}{\log n}.
\end{equation}
Then, substituting the equality \eqref{eq-lower-shift} into the inequality \eqref{ineq-want-neg}, we obtain the desired inequality.
\end{proof}

The following is a polynomial analogue of Proposition \ref{shift-stab-lower}.

\begin{prop}\label{pol-shift-stab-lower}
Let $\mathcal{D}$ be a saturated triangulated category and $\sigma \in \mathrm{Stab}_\Gamma^\dagger(\mathcal{D})$.
For any endofunctor $\Phi : \mathcal{D} \to \mathcal{D}$ with $h_\mathrm{cat}(\Phi)=0$, we have
\begin{gather*}
h_{\sigma,t}^\mathrm{pol}(\Phi) \geq h_\sigma^\mathrm{pol}(\Phi) + \underline{\nu}^\mathrm{pol}(\Phi) \cdot t \quad (t \geq 0),\\
h_{\sigma,t}^\mathrm{pol}(\Phi) \geq h_\sigma^\mathrm{pol}(\Phi) + \overline{\nu}^\mathrm{pol}(\Phi) \cdot t \quad (t \leq 0).
\end{gather*}
\end{prop}

\begin{proof}
This can be proved by combining the arguments in the proofs of Propositions \ref{shift-stab-lower} and \ref{pol-shift-ent-lower}.
\end{proof}

\subsection{Translation length}

Let us consider the metric space $(\mathrm{Stab}_\Gamma^\dagger(\mathcal{D}),d_B)$.
We denote the element of the quotient space $\mathrm{Stab}_\Gamma^\dagger(\mathcal{D})/\mathbb{C}$ represented by $\sigma \in \mathrm{Stab}_\Gamma^\dagger(\mathcal{D})$ as $\bar{\sigma}$.
As every $\mathbb{C}$-orbit is closed in $\mathrm{Stab}_\Gamma^\dagger(\mathcal{D})$,
\begin{equation*}
\bar{d}_B(\bar{\sigma},\bar{\tau}) \coloneqq \inf_{\alpha \in \mathbb{C}} d_B(\sigma,\tau \cdot \alpha)
\end{equation*}
gives a well-defined metric on $\mathrm{Stab}_\Gamma^\dagger(\mathcal{D})/\mathbb{C}$.
The actions of $\mathrm{Aut}(\mathcal{D})$ and $\widetilde{GL}^+(2,\mathbb{R})$ on $\mathrm{Stab}_\Gamma^\dagger(\mathcal{D})$ descend to ones on $\mathrm{Stab}_\Gamma^\dagger(\mathcal{D})/\mathbb{C}$.
Clearly, they act on $(\mathrm{Stab}_\Gamma^\dagger(\mathcal{D})/\mathbb{C},\bar{d}_B)$ isometrically.
We will study the (stable) translation length of the action of some autoequivalences on $(\mathrm{Stab}_\Gamma^\dagger(\mathcal{D})/\mathbb{C},\bar{d}_B)$.

Let us recall the definitions and their basic properties.

\begin{dfn}\label{tran-leng}
Let $(X,d)$ be a metric space and $f : X \to X$ be an isometry.
The {\em translation length} of $f$ is defined by
\begin{equation*}
\tau(f) \coloneqq \inf_{x \in X} d(x,f(x))
\end{equation*}
and the {\em stable translation length} of $f$ is defined by
\begin{equation*}
\tau^s(f) \coloneqq \lim_{n \to \infty} \frac{d(x,f^n(x))}{n}
\end{equation*}
for some $x \in X$.
\end{dfn}

Recall the following elementary fact.

\begin{lem}
The stable translation length is well-defined in the sense that the limit exists and it does not depend on the choice of $x \in X$.
Moreover $\tau(f) \geq \tau^s(f)$ holds.
\end{lem}

\begin{proof}
The sequence $\{d(x,f^n(x))\}_{n=1}^\infty$ is subadditive.
Indeed
\begin{equation*}
d(x,f^{n+m}(x)) \leq d(x,f^n(x)) + d(f^n(x),f^{n+m}(x)) = d(x,f^n(x)) + d(x,f^m(x)).
\end{equation*}
Therefore the limit exists by Fekete's lemma.

Now fix $x,y \in X$.
Then
\begin{equation*}
d(x,f^n(x)) \leq d(x,y) + d(y,f^n(y)) + d(f^n(y),f^n(x)) = 2d(x,y) + d(y,f^n(y)).
\end{equation*}
Thus it follows that
\begin{equation*}
\lim_{n \to \infty} \frac{d(x,f^n(x))}{n} \leq \lim_{n \to \infty} \frac{d(y,f^n(y))}{n}.
\end{equation*}
Similarly, we obtain the opposite inequality.

Again by Fekete's lemma, $\tau^s(f)$ can be written as
\begin{equation*}
\tau^s(f) = \inf_{n \in \mathbb{N}} \frac{d(x,f^n(x))}{n} \leq d(x,f(x))
\end{equation*}
for any $x \in X$.
This implies that $\tau(f) \geq \tau^s(f)$.
\end{proof}

\subsection{Spectral radius and polynomial growth}

Recall that, in our situation, every $\Phi \in \mathrm{Aut}(\mathcal{D})$ induces a linear map $[\Phi]_\Gamma : \Gamma \otimes_\mathbb{Z} \mathbb{C} \to \Gamma \otimes_\mathbb{Z} \mathbb{C}$.
We will compare the categorical entropy and mass growth of $\Phi \in \mathrm{Aut}(\mathcal{D})$ with the logarithm of the spectral radius of $[\Phi]_\Gamma$.

Let us recall the definition.

\begin{dfn}\label{spec-rad}
Let $V$ be a finite dimensional vector space over $\mathbb{R}$ or $\mathbb{C}$ and $A : V \to V$ be a linear map.
Let $\lambda_1,\dots,\lambda_k \in \mathbb{C}$ be the eigenvalues of $A$.
The {\em spectral radius} $\rho(A)$ of $A$ is defined by
\begin{equation*}
\rho(A) \coloneqq \max\{ \lvert\lambda_1\rvert,\dots,\lvert\lambda_k\rvert \}.
\end{equation*}
\end{dfn}

We will also compare the polynomial categorical entropy and polynomial mass growth of $\Phi \in \mathrm{Aut}(\mathcal{D})$ with the polynomial growth rate of $[\Phi]_\Gamma$.
It is defined as follows.

\begin{dfn}[{\cite[Section 2]{CPR},\cite[Definition 4.1]{FFO}}]\label{pol-grow}
Let $V$ be a finite dimensional vector space over $\mathbb{R}$ or $\mathbb{C}$ and $A : V \to V$ be a linear map with $\rho(A) \neq 0$.
Fix a matrix norm $\lVert - \rVert$.
The {\em polynomial growth rate} $s(A)$ of $A$ is defined by
\begin{equation*}
s(A) \coloneqq \lim_{n \to \infty} \frac{\log\lVert A^n \rVert - n\log\rho(A)}{\log n}.
\end{equation*}
\end{dfn}

\begin{lem}[{\cite[Lemma 4.2]{FFO}}]\label{jordan}
The polynomial growth rate is well-defined in the sense that the limit exists and it does not depend on the choice of a matrix norm $\lVert - \rVert$.
Moreover $s(A)+1$ coincides with the maximal size of the Jordan blocks of $A$ whose eigenvalues have absolute value $\rho(A)$.
In particular, $s(A)$ is a non-negative integer.
\end{lem}

For the linear map $[\Phi]_\Gamma : \Gamma \otimes_\mathbb{Z} \mathbb{C} \to \Gamma \otimes_\mathbb{Z} \mathbb{C}$ induced by $\Phi \in \mathrm{Aut}(\mathcal{D})$, we denote its spectral radius (resp. polynomial growth rate) simply by $\rho(\Phi) \coloneqq \rho([\Phi]_\Gamma)$ (resp. $s(\Phi) \coloneqq s([\Phi]_\Gamma)$).

\subsection{Yomdin type inequalities}\label{yom-sec}

The Yomdin type inequalities (see Theorem \ref{yom-thm}) for the (polynomial) categorical entropy and mass growth have been proved for the following cases.

\begin{prop}[{\cite[Theorem 2.13]{KST}}]\label{yom-cat-ent}
Let $\mathcal{D}$ be a numerically finite saturated triangulated category and $\Phi : \mathcal{D} \to \mathcal{D}$ be an autoequivalence.
Then we have
\begin{equation*}
h_\mathrm{cat}(\Phi) \geq \log\rho(\Phi).
\end{equation*}
\end{prop}

\begin{prop}[{\cite[Proposition 4.3]{FFO}}]\label{yom-pol-cat-ent}
Let $\mathcal{D}$ be a numerically finite saturated triangulated category and $\Phi : \mathcal{D} \to \mathcal{D}$ be an autoequivalence such that $h_\mathrm{cat}(\Phi) = \log\rho(\Phi)$.
Then we have
\begin{equation*}
h_\mathrm{cat}^\mathrm{pol}(\Phi) \geq s(\Phi).
\end{equation*}
\end{prop}

\begin{prop}[{\cite[Proposition 3.11]{Ike}}]\label{yom-mass-grow}
Let $\mathcal{D}$ be a triangulated category and $\sigma \in \mathrm{Stab}_\Gamma^\dagger(\mathcal{D})$.
Let $\Phi : \mathcal{D} \to \mathcal{D}$ be an autoequivalence.
Then we have
\begin{equation*}
h_\sigma(\Phi) \geq \log\rho(\Phi).
\end{equation*}
\end{prop}

\begin{prop}[{\cite[Proposition 4.5]{FFO}}]\label{yom-pol-mass-grow}
Let $\mathcal{D}$ be a triangulated category and $\sigma \in \mathrm{Stab}_\Gamma^\dagger(\mathcal{D})$.
Let $\Phi : \mathcal{D} \to \mathcal{D}$ be an autoequivalence such that $h_\sigma(\Phi) = \log\rho(\Phi)$.
Then we have
\begin{equation*}
h_\sigma^\mathrm{pol}(\Phi) \geq s(\Phi).
\end{equation*}
\end{prop}

In the next section, we will show that the opposite inequalities, i.e., the Gromov type inequalities (see Theorem \ref{gro-thm}), also hold for a certain class of autoequivalences.

\section{Gromov--Yomdin type theorems}

\subsection{Compatible triples}

In this section, we shall show that a certain class of autoequivalences satisfies the Gromov--Yomdin type theorems.
Our main notion to consider is the following.

\begin{dfn}\label{triple}
Let $\mathcal{D}$ be a triangulated category.
A triple $(\Phi,\sigma,g) \in \mathrm{Aut}(\mathcal{D}) \times \mathrm{Stab}_\Gamma^\dagger(\mathcal{D}) \times \widetilde{GL}^+(2,\mathbb{R})$ is called {\em compatible} if it satisfies
\begin{equation*}
\Phi \cdot \sigma = \sigma \cdot g.
\end{equation*}
\end{dfn}

\begin{rmk}\label{triple-rmk}
For a given $\Phi \in \mathrm{Aut}(\mathcal{D})$, there exist $\sigma \in \mathrm{Stab}_\Gamma^\dagger(\mathcal{D})$ and $g \in \widetilde{GL}^+(2,\mathbb{R})$ such that the triple $(\Phi,\sigma,g)$ is compatible if and only if the action of $\Phi$ on the quotient space $\mathrm{Stab}_\Gamma^\dagger(\mathcal{D})/\widetilde{GL}^+(2,\mathbb{R})$ has a fixed point.

As mentioned in Section \ref{intro}, an autoequivalence $\Phi \in \mathrm{Aut}(\mathcal{D})$ in a compatible triple $(\Phi,\sigma,g)$ can be thought of as a categorical analogue of a holomorphic automorphism.
\end{rmk}

Let $V$ be a vector space over $\mathbb{R}$.
We call a subset $S \subset V$ an {\em $\mathbb{R}$-spanning set} of $V$ if the linear subspace of $V$ spanned by $S$ coincides with $V$.
For $\sigma = (Z_\sigma,\mathcal{P}_\sigma) \in \mathrm{Stab}_\Gamma^\dagger(\mathcal{D})$, we say the central charge $Z_\sigma : \Gamma \to \mathbb{C}$ has {\em spanning image} if the set
\begin{equation*}
\{Z_\sigma(E) \,|\, E \text{ is } \sigma \text{-semistable} \}
\end{equation*}
is an $\mathbb{R}$-spanning set of $\mathbb{C} \cong \mathbb{R}^2$.

The following theorem says that if a triple $(\Phi,\sigma,g) \in \mathrm{Aut}(\mathcal{D}) \times \mathrm{Stab}_\Gamma^\dagger(\mathcal{D}) \times \widetilde{GL}^+(2,\mathbb{R})$ is compatible and $Z_\sigma$ has spanning image then the Gromov--Yomdin type theorems hold for $\Phi$.

\begin{thm}\label{cat-gy-thm}
Let $(\Phi,\sigma,g) \in \mathrm{Aut}(\mathcal{D}) \times \mathrm{Stab}_\Gamma^\dagger(\mathcal{D}) \times \widetilde{GL}^+(2,\mathbb{R})$ be a compatible triple such that $Z_\sigma$ has spanning image.
Then we have the following:
\begin{enumerate}
\item $h_\sigma(\Phi) = \log\rho(\Phi) = \log\rho(M_g)$.
\item $h_\sigma^\mathrm{pol}(\Phi) = s(\Phi) = s(M_g)$.
\end{enumerate}
\end{thm}

\begin{proof}
(1)
Note that the compatibility of the triple $(\Phi,\sigma,g)$ implies that
\begin{equation}\label{gy-cond}
Z_\sigma \circ [\Phi]_\Gamma = M_g \cdot Z_\sigma
\end{equation}
and that $\Phi$ sends $\sigma$-semistable objects (of phase $\phi$) to $\sigma$-semistable objects (of phase $f_g(\phi)$).
Take any object $0 \neq E \in \mathcal{D}$ then
\begin{equation*}
v(E) = \sum_{i=1}^k v(E_i) \in \Gamma
\end{equation*}
where $E_1,\dots,E_k$ are the semistable factors of the Harder--Narasimhan filtration of $E$ with respect to $\sigma$.
Then, for every $n \in \mathbb{N}$, $\Phi^nE_1,\dots,\Phi^nE_k$ are the semistable factors of the Harder--Narasimhan filtration of $\Phi^nE$ with respect to $\sigma$.
Therefore we have
\begin{equation*}
m_\sigma(\Phi^nE) = \sum_{i=1}^k \lvert Z_\sigma(\Phi^nE_i) \rvert = \sum_{i=1}^k \lvert M_g^n Z_\sigma(E_i) \rvert \leq \lVert M_g^n \rVert \sum_{i=1}^k \lvert Z_\sigma(E_i) \rvert = \lVert M_g^n \rVert m_\sigma(E)
\end{equation*}
where $\lVert - \rVert$ is the matrix norm induced by the standard vector norm on $\mathbb{R}^2$.
Then, by Gelfand's formula, we obtain
\begin{equation*}
h_\sigma(\Phi) = \sup_{0 \neq E \in \mathcal{D}} \left( \limsup_{n \to \infty} \frac{\log m_{\sigma,t}(\Phi^nE)}{n} \right) \leq \lim_{n \to \infty} \frac{\log \lVert M_g^n \rVert}{n} = \log\rho(M_g).
\end{equation*}

Now let $\mu_\Phi(t)$ be the minimal polynomial of $[\Phi]_\Gamma$.
Then, by Cayley--Hamilton theorem and the equation \eqref{gy-cond}, we have
\begin{equation*}
0 = Z_\sigma(\mu_\Phi([\Phi]_\Gamma)v) = \mu_\Phi(M_g)Z_\sigma(v)
\end{equation*}
for every $v \in \Gamma$.
From this, we see that $\mu_\Phi(M_g) = 0$ as $Z_\sigma$ is assumed to have spanning image.
This implies that the minimal polynomial of $M_g$ divides $\mu_\Phi$ and hence every eigenvalue of $M_g$ is an eigenvalue of $[\Phi]_\Gamma$.
Therefore we have
\begin{equation*}
\log\rho(M_g) \leq \log\rho(\Phi) \leq h_\sigma(\Phi)
\end{equation*}
where the second inequality follows from Proposition \ref{yom-mass-grow}.

(2)
Note that $h_\sigma(\Phi) = \log\rho(M_g)$ by (1).
As in the proof of (1), we obtain
\begin{align*}
h_\sigma^\mathrm{pol}(\Phi)
&= \sup_{0 \neq E \in \mathcal{D}} \limsup_{n \to \infty} \frac{\log m_\sigma(\Phi^nE) - nh_\sigma(\Phi)}{\log n}\\
&\leq \limsup_{n \to \infty} \frac{\log\lVert M_g^n \rVert - n\log\rho(M_g)}{\log n}\\
&= s(M_g).
\end{align*}
Since $s(M_g)+1$ (resp. $s(\Phi)+1$) coincides with the maximal size of the Jordan blocks of $M_g$ (resp. $[\Phi]_\Gamma$) whose eigenvalues have absolute value $\rho(M_g)=\rho(\Phi)$ by Lemma \ref{jordan}, the minimal polynomial argument as in the proof of (1) shows that
\begin{equation*}
s(M_g) \leq s(\Phi) \leq h_\sigma^\mathrm{pol}(\Phi)
\end{equation*}
where the second inequality follows from Proposition \ref{yom-pol-mass-grow}.
\end{proof}

Combining Theorems \ref{mass-grow-thm} (2), \ref{pol-mass-grow-thm} (2) and \ref{cat-gy-thm}, we immediately obtain the following corollary.

\begin{cor}
Let $(\Phi,\sigma,g) \in \mathrm{Aut}(\mathcal{D}) \times \mathrm{Stab}_\Gamma^\dagger(\mathcal{D}) \times \widetilde{GL}^+(2,\mathbb{R})$ be a compatible triple such that $Z_\sigma$ has spanning image.
Assume that $\sigma$ lies in a connected component containing an algebraic stability condition.
Then we have the following:
\begin{enumerate}
\item $h_\mathrm{cat}(\Phi) = \log\rho(\Phi) = \log\rho(M_g)$.
\item $h_\mathrm{cat}^\mathrm{pol}(\Phi) = s(\Phi) = s(M_g)$.
\end{enumerate}
\end{cor}

\subsection{Saturated case}

Let $\mathcal{D}$ be a saturated triangulated category and $(\Phi,\sigma,g) \in \mathrm{Aut}(\mathcal{D}) \times \mathrm{Stab}_\Gamma^\dagger(\mathcal{D}) \times \widetilde{GL}^+(2,\mathbb{R})$ be a compatible triple.
In this section, we show that the categorical entropy $h_t(\Phi)$ and the mass growth $h_{\sigma,t}(\Phi)$ (resp. their polynomial analogues) are linear in $t$ where the slopes are given by the shifting number.
As a corollary, we will see that the Gromov--Yomdin type theorem for the categorical entropy (resp. polynomial categorical entropy) of $\Phi$ holds if and only if $h_t(\Phi) = h_{\sigma,t}(\Phi)$ (resp. $h_t^\mathrm{pol}(\Phi) = h_{\sigma,t}^\mathrm{pol}(\Phi)$) for all $t$.

One of the key ingredients is the following.

\begin{prop}[{\cite[Proposition 6.17]{KP}}]\label{shift-prop}
Suppose $\mathcal{D}$ is saturated and let $(\Phi,\sigma,g) \in \mathrm{Aut}(\mathcal{D}) \times \mathrm{Stab}_\Gamma^\dagger(\mathcal{D}) \times \widetilde{GL}^+(2,\mathbb{R})$ be a compatible triple.
Then we have $\overline{\nu}(\Phi) = \underline{\nu}(\Phi)$.
\end{prop}

We can prove a similar statement for the polynomial case.

\begin{prop}\label{pol-shift-prop}
Suppose $\mathcal{D}$ is saturated and let $(\Phi,\sigma,g) \in \mathrm{Aut}(\mathcal{D}) \times \mathrm{Stab}_\Gamma^\dagger(\mathcal{D}) \times \widetilde{GL}^+(2,\mathbb{R})$ be a compatible triple such that $h_\mathrm{cat}(\Phi) = 0$.
Then we have $\overline{\nu}^\mathrm{pol}(\Phi) = \underline{\nu}^\mathrm{pol}(\Phi)$.
\end{prop}

\begin{proof}
By Proposition \ref{shift-prop}, we have $\overline{\nu}(\Phi) = \underline{\nu}(\Phi)$.
Then, using the inequalities \eqref{ineq1} and \eqref{ineq2}, we obtain
\begin{align*}
0
&\leq \overline{\nu}^\mathrm{pol}(\Phi) - \underline{\nu}^\mathrm{pol}(\Phi)\\
&\leq \limsup_{n \to \infty} \frac{\epsilon^+(G,\Phi^nG) - n\overline{\nu}(\Phi)}{\log n} - \liminf_{n \to \infty} \frac{\epsilon^-(G,\Phi^nG) - n\underline{\nu}(\Phi)}{\log n}\\
&= \limsup_{n \to \infty} \frac{\epsilon^+(G,\Phi^nG) - \epsilon^-(G,\Phi^nG)}{\log n}\\
&\leq \limsup_{n \to \infty} \frac{\phi_\sigma^+(\Phi^nG) - \phi_\sigma^-(\Phi^nG) + \phi_\sigma^+(SG) - \phi_\sigma^-(G)}{\log n}\\
&= \limsup_{n \to \infty} \frac{f_g^n\phi_\sigma^+(G) - f_g^n\phi_\sigma^-(G)}{\log n}.
\end{align*}
In general, for any increasing function $f : \mathbb{R} \to \mathbb{R}$ such that $f(\phi+1) = f(\phi)$, the proof of \cite[Proposition 6.17]{KP} shows that
\begin{equation*}
\lim_{n \to \infty} \frac{f^n(\phi_1) - f^n(\phi_2)}{\log n} = 0
\end{equation*}
for any $\phi_1,\phi_2 \in \mathbb{R}$.
This shows that $\overline{\nu}^\mathrm{pol}(\Phi) = \underline{\nu}^\mathrm{pol}(\Phi)$.
\end{proof}

\begin{prop}\label{linearity}
Suppose $\mathcal{D}$ is saturated and let $(\Phi,\sigma,g) \in \mathrm{Aut}(\mathcal{D}) \times \mathrm{Stab}_\mathcal{N}^\dagger(\mathcal{D}) \times \widetilde{GL}^+(2,\mathbb{R})$ be a compatible triple.
Then we have the following:
\begin{enumerate}
\item $h_t(\Phi) = h_\mathrm{cat}(\Phi) + \overline{\nu}(\Phi) \cdot t$.
\item $h_{\sigma,t}(\Phi) = h_\sigma(\Phi) + \overline{\nu}(\Phi) \cdot t$.
\item $h_t^\mathrm{pol}(\Phi) = h_\mathrm{cat}^\mathrm{pol}(\Phi) + \overline{\nu}^\mathrm{pol}(\Phi) \cdot t$ under the assumption that $h_\mathrm{cat}(\Phi) = 0$.
\item $h_{\sigma,t}^\mathrm{pol}(\Phi) = h_\sigma^\mathrm{pol}(\Phi) + \overline{\nu}^\mathrm{pol}(\Phi) \cdot t$ under the assumption that $h_\mathrm{cat}(\Phi) = 0$.
\end{enumerate}
\end{prop}

\begin{proof}
(1)
By Propositions \ref{shift-ent} and \ref{shift-ent-lower}, we have
\begin{gather*}
h_\mathrm{cat}(\Phi) + \underline{\nu}(\Phi) \cdot t \leq h_t(\Phi) \leq h_\mathrm{cat}(\Phi) + \overline{\nu}(\Phi) \cdot t \quad (t \geq 0),\\
h_\mathrm{cat}(\Phi) + \overline{\nu}(\Phi) \cdot t \leq h_t(\Phi) \leq h_\mathrm{cat}(\Phi) + \underline{\nu}(\Phi) \cdot t \quad (t \leq 0).
\end{gather*}
As $\overline{\nu}(\Phi) = \underline{\nu}(\Phi)$ by Proposition \ref{shift-prop}, we see that $h_t(\Phi) = h_\mathrm{cat}(\Phi) + \overline{\nu}(\Phi) \cdot t$ for all $t \in \mathbb{R}$.

(2) Similarly, apply Propositions \ref{shift-stab}, \ref{shift-stab-lower} and \ref{shift-prop}.

(3) Similarly, apply Propositions \ref{pol-shift-ent}, \ref{pol-shift-ent-lower} and \ref{pol-shift-prop}.

(4) Similarly, apply Propositions \ref{pol-shift-stab}, \ref{pol-shift-stab-lower} and \ref{pol-shift-prop}.
\end{proof}

As a corollary, we obtain a necessary and sufficient condition for the Gromov--Yomdin type theorem for the (polynomial) categorical entropy.

\begin{cor}\label{iff-cor}
Suppose $\mathcal{D}$ is saturated and let $(\Phi,\sigma,g) \in \mathrm{Aut}(\mathcal{D}) \times \mathrm{Stab}_\mathcal{N}^\dagger(\mathcal{D}) \times \widetilde{GL}^+(2,\mathbb{R})$ be a compatible triple such that $Z_\sigma$ has spanning image.
Then we have the following:
\begin{enumerate}
\item $h_\mathrm{cat}(\Phi) = \log\rho(\Phi)$ if and only if $h_t(\Phi) = h_{\sigma,t}(\Phi)$ for all $t \in \mathbb{R}$.
\item $h_\mathrm{cat}^\mathrm{pol}(\Phi) = s(\Phi)$ if and only if $h_t^\mathrm{pol}(\Phi) = h_{\sigma,t}^\mathrm{pol}(\Phi)$ for all $t \in \mathbb{R}$ under the assumption that $h_\mathrm{cat}(\Phi) = 0$.
\end{enumerate}
\end{cor}

\begin{proof}
(1)
Suppose $h_t(\Phi) = h_{\sigma,t}(\Phi)$ for all $t \in \mathbb{R}$, in particular, $h_\mathrm{cat}(\Phi) = h_\sigma(\Phi)$.
Then Theorem \ref{cat-gy-thm} (1) shows that $h_\mathrm{cat}(\Phi) = h_\sigma(\Phi) = \log\rho(\Phi)$.

Conversely, suppose $h_\mathrm{cat}(\Phi) = \log\rho(\Phi)$.
By Theorem \ref{cat-gy-thm} (1), we have $\log\rho(\Phi) = h_\sigma(\Phi)$.
Thus the assertion readily follows from Proposition \ref{linearity} (1), (2).

(2)
Similarly, apply Theorem \ref{cat-gy-thm} (2) and Proposition \ref{linearity} (3), (4).
\end{proof}

\subsection{Comparison with translation length}

Consider a triangulated category $\mathcal{D}$ (not necessarily saturated) and a compatible triple $(\Phi,\sigma,g) \in \mathrm{Aut}(\mathcal{D}) \times \mathrm{Stab}_\Gamma^\dagger(\mathcal{D}) \times \widetilde{GL}^+(2,\mathbb{R})$.
In this section, we study the relationship between the mass growth of $\Phi$ and the stable translation length of the action of $\Phi$ on $(\mathrm{Stab}_\Gamma^\dagger(\mathcal{D})/\mathbb{C},\bar{d}_B)$.

We need the following linear algebraic lemma.

\begin{lem}\label{norm-lem}
Let $V$ be a vector space over $\mathbb{R}$ and $S \subset V$ be an $\mathbb{R}$-spanning set of $V$.
For a linear map $A : V \to V$, define
\begin{equation*}
\lVert A \rVert_S \coloneqq \sup_{v \in S} \frac{\lvert Av \rvert}{\lvert v \rvert}.
\end{equation*}
Then $\lVert - \rVert_S$ is a matrix norm.
Moreover, if $A : V \to V$ is a linear isomorphism acting bijectively on $S$, we have
\begin{equation*}
\lVert A^{-1} \rVert_S = \left( \inf_{v \in S} \frac{\lvert Av \rvert}{\lvert v \rvert} \right)^{-1}.
\end{equation*}
\end{lem}

\begin{proof}
The proof that $\lVert - \rVert_S$ is a matrix norm is standard.
Let us prove the latter assertion.
\begin{equation*}
\lVert A^{-1} \rVert_S = \sup_{v \in S} \frac{\lvert A^{-1}v \rvert}{\lvert v \rvert} = \sup_{v \in S} \frac{\lvert v \rvert}{\lvert Av \rvert} = \left(\left( \sup_{v \in S} \frac{\lvert v \rvert}{\lvert Av \rvert} \right)^{-1}\right)^{-1} = \left( \inf_{v \in S} \frac{\lvert Av \rvert}{\lvert v \rvert} \right)^{-1}
\end{equation*}
where the second equality holds by the assumption that $A$ acts bijectively on $S$.
\end{proof}

We can show the following by generalizing the proof of \cite[Theorem 4.9]{Kik2} (also see \cite[Proposition 4.1]{Woo1}).

\begin{thm}\label{tran-leng-thm}
Let $(\Phi,\sigma,g) \in \mathrm{Aut}(\mathcal{D}) \times \mathrm{Stab}_\Gamma^\dagger(\mathcal{D}) \times \widetilde{GL}^+(2,\mathbb{R})$ be a compatible triple.
Then we have
\begin{equation*}
\tau^s(\Phi) = \log\frac{\rho(M_g)}{\det(M_g)^{1/2}}.
\end{equation*}
In particular, if $Z_\sigma$ has spanning image and $\det(M_g)=1$ then
\begin{equation*}
\tau^s(\Phi) = \log\rho(M_g) = \log\rho(\Phi) = h_\sigma(\Phi).
\end{equation*}
\end{thm}

\begin{proof}
(1)
Let us first show that $\tau^s(\Phi) \leq \log\frac{\rho(M_g)}{\det(M_g)^{1/2}}$.
Since $\sigma$ and $\sigma \cdot g$ have the same sets of semistable objects and the same Harder--Narasimhan filtrations, for any $n \in \mathbb{N}$ and $\alpha \in \mathbb{C}$, we have
\begin{align*}
d_B(\sigma,\Phi^n \cdot \sigma \cdot \alpha)
&= d_B(\sigma,\sigma \cdot (g^n\alpha))\\
&= \sup_{0 \neq E \in \mathcal{D}} \left\{ \lvert \phi_{\sigma \cdot (g^n\alpha)}^\pm(E) - \phi_\sigma^\pm(E) \rvert,\left\lvert\log\frac{m_{\sigma \cdot (g^n\alpha)}(E)}{m_\sigma(E)}\right\rvert \right\}\\
&= \sup_{E:\, \sigma\text{-s.s.}} \left\{ \lvert \phi_{\sigma \cdot (g^n\alpha)}(E) - \phi_\sigma(E) \rvert,\left\lvert\log\frac{\lvert Z_{\sigma \cdot (g^n\alpha)}(E) \rvert}{\lvert Z_\sigma(E) \rvert}\right\rvert \right\}\\
&= \sup_{E:\, \sigma\text{-s.s.}} \left\{ \lvert f_{g^n\alpha}(\phi_\sigma(E)) - \phi_\sigma(E) \rvert,\left\lvert\log\frac{\lvert M_{g^n\alpha}^{-1}Z_\sigma(E) \rvert}{\lvert Z_\sigma(E) \rvert}\right\rvert \right\}\\
&= \max \left\{ \sup_{E:\, \sigma\text{-s.s.}} \lvert f_{g^n\alpha}(\phi_\sigma(E)) - \phi_\sigma(E) \rvert,\sup_{E:\, \sigma\text{-s.s.}} \left\lvert\log\frac{\lvert M_{g^n\alpha}^{-1}Z_\sigma(E) \rvert}{\lvert Z_\sigma(E) \rvert}\right\rvert \right\}\\
&\leq \max \left\{ \sup_{\phi \in \mathbb{R}} \lvert f_{g^n\alpha}(\phi) - \phi \rvert,\sup_{v \in \mathbb{R}^2,\lVert v \rVert = 1} \left\lvert\log\lvert M_{g^n\alpha}^{-1}v \rvert\right\rvert \right\}\\
&= \max \{A_{n,\alpha},B_{n,\alpha}\}
\end{align*}
where
\begin{gather*}
A_{n,\alpha} \coloneqq \sup_{\phi \in \mathbb{R}} \lvert f_{g^n\alpha}(\phi) - \phi \rvert,\\
B_{n,\alpha} \coloneqq \sup_{v \in \mathbb{R}^2,\lVert v \rVert = 1} \left\lvert\log\lvert M_{g^n\alpha}^{-1}v \rvert\right\rvert.
\end{gather*}

\begin{itemize}
\item
Since $f_g^n - \mathrm{id} : \mathbb{R} \to \mathbb{R}$ is a 1-periodic function,
\begin{equation*}
A_{n,\alpha} = \sup_{\phi \in \mathbb{R}} \lvert f_g^n(\phi) - \phi + \Re\alpha \rvert = \max_{\phi \in [0,1]} \lvert f_g^n(\phi) - \phi + \Re\alpha \rvert.
\end{equation*}
Thus, for every $\alpha \in \mathbb{C}$ such that $\Re\alpha = -f_g^n(0)$,
\begin{equation}\label{estimate-a}
A_{n,\alpha} = \max_{\phi \in [0,1]} \lvert f_g^n(\phi) - \phi -f_g^n(0) \rvert < 1.
\end{equation}
\item
Let $M_g' \coloneqq \frac{1}{\det(M_g)^{1/2}} M_g \in SL(2,\mathbb{R})$ and $\lVert - \rVert$ be the matrix norm induced by the standard vector norm on $\mathbb{R}^2$.
Then
\begin{align*}
B_{n,\alpha}
&= \sup_{v \in \mathbb{R}^2,\lVert v \rVert = 1} \left\{\pm\log \lvert M_{g^n\alpha}^{-1}v \rvert \right\}\\
&= \max \left\{\log\sup_{v \in \mathbb{R}^2,\lVert v \rVert = 1} \lvert M_{g^n\alpha}^{-1}v \rvert,\log\left(\inf_{v \in \mathbb{R}^2,\lVert v \rVert = 1} \lvert M_{g^n\alpha}^{-1}v \rvert\right)^{-1}\right\}\\
&= \max \{ \log\lVert M_{g^n\alpha}^{-1} \rVert,\log\lVert M_{g^n\alpha} \rVert\}\\
&= \max \{ \log\lVert M_g^{-n} \rVert - \pi\Im\alpha,\log\lVert M_g^n \rVert + \pi\Im\alpha\}\\
&= \max \{ \log\lVert M_g'^{\pm n} \rVert \pm \log\det(M_g)^{n/2} \pm \pi\Im\alpha\}.
\end{align*}
Therefore, for every $\alpha \in \mathbb{C}$ such that $\Im\alpha = -\frac{1}{\pi}\log\det(M_g)^{n/2}$,
\begin{equation}\label{estimate-b}
B_{n,\alpha} = \max\{ \log\lVert M_g'^{-n} \rVert,\log\lVert M_g'^n \rVert \}.
\end{equation}
\end{itemize}

By the above estimates, we obtain
\begin{align*}
\frac{\bar{d}_B(\bar{\sigma},\Phi^n \cdot \bar{\sigma})}{n}
&= \frac{1}{n} \inf_{\alpha \in \mathbb{C}} d_B(\sigma,\Phi^n \cdot \sigma \cdot \alpha)\\
&\leq \frac{1}{n} \inf_{\alpha \in \mathbb{C}} \max\{A_{n,\alpha},B_{n,\alpha}\}\\
&\leq \frac{1}{n} \max\{ 1,\log\lVert M_g'^{-n} \rVert,\log\lVert M_g'^n \rVert \}\\
&= \max\left\{ \frac{1}{n},\log\lVert M_g'^{-n} \rVert^{1/n},\log\lVert M_g'^n \rVert^{1/n} \right\}
\end{align*}
where the third inequality follows from the inequalities \eqref{estimate-a} and \eqref{estimate-b} for $\alpha = -f_g^n(0)-\frac{i}{\pi}\log\det(M_g)^{n/2}$.
Since $\lVert - \rVert$ is a matrix norm, for every $\varepsilon>0$, there exists $N \in \mathbb{N}$ such that
\begin{equation*}
\rho(M_g')-\varepsilon < \lVert M_g'^{\pm n} \rVert^{1/n} < \rho(M_g')+\varepsilon
\end{equation*}
for every $n>N$.
Note that $\rho(M_g') = \rho(M_g'^{-1}) \geq 1$.
Therefore, if we take $N$ sufficiently large,
\begin{equation*}
\frac{\bar{d}_B(\bar{\sigma},\Phi^n \cdot \bar{\sigma})}{n} < \log(\rho(M_g')+\varepsilon) = \log\left(\frac{\rho(M_g)}{\det(M_g)^{1/2}} + \varepsilon\right)
\end{equation*}
for every $n>N$.
Since $\varepsilon>0$ is arbitrary, it follows that
\begin{equation*}
\tau^s(\Phi) = \lim_{n \to \infty} \frac{\bar{d}_B(\bar{\sigma},\Phi^n \cdot \bar{\sigma})}{n} \leq \log\frac{\rho(M_g)}{\det(M_g)^{1/2}}.
\end{equation*}

(2)
Let us next show that $\tau^s(\Phi) \geq \log\frac{\rho(M_g)}{\det(M_g)^{1/2}}$ assuming that $Z_\sigma$ has spanning image.
As in (1), for any $n \in \mathbb{N}$ and $\alpha \in \mathbb{C}$, we have
\begin{equation*}
d_B(\sigma,\Phi^n \cdot \sigma \cdot \alpha) = \max \{A_{n,\alpha}',B_{n,\alpha}'\}
\end{equation*}
where
\begin{gather*}
A_{n,\alpha}' \coloneqq \sup_{E:\, \sigma\text{-s.s.}} \lvert f_{g^n\alpha}(\phi_\sigma(E)) - \phi_\sigma(E) \rvert,\\
B_{n,\alpha}' \coloneqq \sup_{E:\, \sigma\text{-s.s.}} \left\lvert\log\frac{\lvert M_{g^n\alpha}^{-1}Z_\sigma(E) \rvert}{\lvert Z_\sigma(E) \rvert}\right\rvert.
\end{gather*}

As in (1), we obtain the following.

\begin{itemize}
\item
Again let $M_g' \coloneqq \frac{1}{\det(M_g)^{1/2}} M_g \in SL(2,\mathbb{R})$ and $\lVert - \rVert_S$ be the matrix norm determined by the set $S \coloneqq \{Z_\sigma(E) \,|\, E \text{ is } \sigma \text{-semistable} \}$ by Lemma \ref{norm-lem}.
Then, as before,
\begin{equation}\label{estimate-b'}
B_{n,\alpha}' = \max \{ \log\lVert M_g'^{\pm n} \rVert_S \pm \log\det(M_g)^{n/2} \pm \pi\Im\alpha\}.
\end{equation}
\end{itemize}

Now let us prove the claim.
If $\rho(M_g')=1$ then
\begin{equation*}
0 \leq \lim_{n \to \infty} \frac{\bar{d}_B(\bar{\sigma},\Phi^n \cdot \bar{\sigma})}{n} \leq \log\frac{\rho(M_g)}{\det(M_g)^{1/2}} = \log\rho(M_g') = 0.
\end{equation*}
Assume that $\rho(M_g')>1$.
Then, using the above estimates,
\begin{align*}
\frac{\bar{d}_B(\bar{\sigma},\Phi^n \cdot \bar{\sigma})}{n}
&= \frac{1}{n} \inf_{\alpha \in \mathbb{C}} d_B(\sigma,\Phi^n \cdot \sigma \cdot \alpha)\\
&= \frac{1}{n} \inf_{\alpha \in \mathbb{C}} \max\{A_{n,\alpha}',B_{n,\alpha}'\}\\
&\geq \frac{1}{n} \inf_{\alpha \in \mathbb{C}} B_{n,\alpha}'\\
&= \frac{1}{n} \inf_{\alpha \in \mathbb{C}} \max \{ \log\lVert M_g'^{\pm n} \rVert_S \pm \log\det(M_g)^{n/2} \pm \pi\Im\alpha\}\\
&= \inf_{\alpha \in \mathbb{C}} \max \left\{ \log\lVert M_g'^{\pm n} \rVert_S^{1/n} \pm \log\det(M_g)^{1/2} \pm \frac{\pi\Im\alpha}{n}\right\}.
\end{align*}
For any $\rho(M_g')-1>\varepsilon>0$, there exists $N \in \mathbb{N}$ such that
\begin{equation*}
\rho(M_g')-\varepsilon < \lVert M_g'^{\pm n} \rVert_S^{1/n} < \rho(M_g')+\varepsilon
\end{equation*}
for every $n>N$.
Hence we get
\begin{align*}
\frac{\bar{d}_B(\bar{\sigma},\Phi^n \cdot \bar{\sigma})}{n}
&> \inf_{\alpha \in \mathbb{C}}\max \left\{ \log(\rho(M_g')-\varepsilon) \pm \log\det(M_g)^{1/2} \pm \frac{\pi\Im\alpha}{n} \right\}\\
&= \log(\rho(M_g')-\varepsilon)\\
&= \log\left(\frac{\rho(M_g)}{\det(M_g)^{1/2}} - \varepsilon\right).
\end{align*}
Therefore
\begin{equation*}
\tau^s(\Phi) = \lim_{n \to \infty} \frac{\bar{d}_B(\bar{\sigma},\Phi^n \cdot \bar{\sigma})}{n} \geq \log\frac{\rho(M_g)}{\det(M_g)^{1/2}}
\end{equation*}
as $\varepsilon>0$ is arbitrary.

(3)
Finally, let us show that $\tau^s(\Phi) \geq \log\frac{\rho(M_g)}{\det(M_g)^{1/2}}$ assuming that $Z_\sigma$ does not have spanning image.
In this case, the relation $Z_\sigma \circ [\Phi]_\Gamma = M_g \cdot Z_\sigma$ forces
\begin{equation*}
M_g'Z_\sigma(E) = Z_\sigma(\Phi E) = \lambda Z_\sigma(E)
\end{equation*}
for every $\sigma$-semistable object $E$ and some $\lambda \in \mathbb{R}$ which necessarily is an eigenvalue of $M_g'$.
This then implies that $\lVert M_g'^{\pm n} \rVert_S = \lvert\lambda\rvert^{\pm n}$ and therefore
\begin{align*}
\frac{\bar{d}_B(\bar{\sigma},\Phi^n \cdot \bar{\sigma})}{n}
&\geq \inf_{\alpha \in \mathbb{C}} \max \left\{ \log\lVert M_g'^{\pm n} \rVert_S^{1/n} \pm \log\det(M_g)^{1/2} \pm \frac{\pi\Im\alpha}{n}\right\}\\
&= \inf_{\alpha \in \mathbb{C}} \max \left\{ \log\lvert \lambda \rvert^{\pm 1} \pm \log\det(M_g)^{1/2} \pm \frac{\pi\Im\alpha}{n}\right\}\\
&= \max\{\log\lvert \lambda \rvert^{\pm 1}\}\\
&= \log\rho(M_g')\\
&= \log\frac{\rho(M_g)}{\det(M_g)^{1/2}}.
\end{align*}
Here the fourth equality holds since $\lambda,\lambda^{-1}$ are the eigenvalues of $M_g' \in SL(2,\mathbb{R})$.
\end{proof}

\section{Remarks on compatible triples}

\subsection{Criterion for compatible triples}

In this section, we give a criterion for a triple $(\Phi,\sigma,g) \in \mathrm{Aut}(\mathcal{D}) \times \mathrm{Stab}_\Gamma^\dagger(\mathcal{D}) \times \widetilde{GL}^+(2,\mathbb{R})$ to be compatible, i.e., $\Phi \cdot \sigma = \sigma \cdot g$.

First of all, note that the compatibility implies the following:
\begin{enumerate}
\item $Z_\sigma \circ [\Phi]_\Gamma = M_g \cdot Z_\sigma$.\label{condition-1}
\item $\Phi(\mathcal{P}_\sigma(\phi)) = \mathcal{P}_\sigma(f_g(\phi))$ for any $\phi \in \mathbb{R}$.\label{condition-2}
\end{enumerate}
The condition \eqref{condition-1} is easier to check in practice and one might wonder whether the condition \eqref{condition-1} implies the condition \eqref{condition-2}.
This is not true in general as the following example shows.

\begin{ex}\label{ginz-ex}
For an odd positive integer $d \neq 1$, let $\Gamma_2^d$ be the $d$-Calabi--Yau Ginzburg dg algebra associated to the $A_2$ quiver $1 \leftarrow 2$ \cite{Gin}.
Let $D_{fd}(\Gamma_2^d)$ be the derived category of dg $\Gamma_2^d$-modules with finite dimensional cohomology and $S_1,S_2$ be the simple modules corresponding to two vertices.
It is known that they are $d$-spherical and satisfy $\mathrm{Hom}_{D_{fd}(\Gamma_2^d)}^\bullet(S_1,S_2) = \mathbb{C}[-1]$ \cite[Lemma 2.15]{KY}.
Denote by $T_{S_1}$ the spherical twist along $S_1$ \cite{ST}.
It sends an object $E$ to $\mathrm{Cone}(\mathrm{Hom}_{D_{fd}(\Gamma_2^d)}^\bullet(S_1,E) \otimes S_1 \overset{\mathrm{ev}}{\to} E)$.

Let $\mathcal{H}$ be the standard algebraic heart of $D_{fd}(\Gamma_2^d)$.
To give a stability condition $\sigma$ with heart $\mathcal{H}$, it is enough to choose $z_i \in \mathbb{H} \coloneqq \{me^{i\pi\phi} \,|\, m \in \mathbb{R}_{>0},\phi \in [0,1)\}$ and put $Z_\sigma(S_i) \coloneqq z_i$ ($i=1,2$).
Let us choose $z_i$'s so that $\arg z_2 > \arg z_1$.
We have
\begin{equation*}
Z_\sigma(T_{S_1}S_1) = Z_\sigma(S_1[1-d]) = (-1)^{1-d}z_1 = z_1,
\end{equation*}
and since $T_{S_1}S_2$ sits in the exact triangle $S_1[-1] \to S_2 \to T_{S_1}S_2 \to S_1$,
\begin{equation*}
Z_\sigma(T_{S_1}S_2) = Z_\sigma(S_1) + Z_\sigma(S_2) = z_1 + z_2.
\end{equation*}
Hence, if we call $M : \mathbb{R}^2 \to \mathbb{R}^2$ the linear isomorphism defined by sending $z_1 \mapsto z_1$ and $z_2 \mapsto z_1+z_2$, we get
\begin{equation*}
Z_\sigma \circ [T_{S_1}] = M \cdot Z_\sigma.
\end{equation*}
As $M \in GL^+(2,\mathbb{R})$, we see that the condition \eqref{condition-1} is satisfied for $T_{S_1}$.

If $T_{S_1}$ further satisfies the condition \eqref{condition-2}, it must send a semistable object to a semistable object.
In particular, $T_{S_1}S_2$ should be semistable.
However, since $T_{S_1}S_2$ sits in the short exact sequence
\begin{equation}\label{ses}
0 \to S_2 \to T_{S_1}S_2 \to S_1 \to 0
\end{equation}
in $\mathcal{H} = \mathcal{P}_\sigma(0,1]$, we should have $\arg z_2 = \arg Z_\sigma(S_2) \leq \arg Z_\sigma(T_{S_1}S_2) \leq \arg Z_\sigma(S_1) = z_1$ which is not the case because of how we chose $z_i$'s.
\end{ex}

\begin{rmk}\label{ginz-rmk}
Actually, the above example shows that there are no $\sigma \in \mathrm{Stab}^\dagger(D_{fd}(\Gamma_2^d))$ with heart $\mathcal{H}$ and $g \in \widetilde{GL}^+(2,\mathbb{R})$ such that $(T_{S_1},\sigma,g)$ is compatible.
Indeed, notice that the compatibility implies that $T_{S_1}$ sends $\mathcal{P}_\sigma(0,1]$ to $\mathcal{P}_\sigma(\psi,\psi+1]$ where $\psi \coloneqq f_g(0)$.
However we have $T_{S_1}S_1 = S_1[1-d] \in \mathcal{P}_\sigma(1-d,2-d]$ whereas $T_{S_1}S_2 \in \mathcal{P}_\sigma(0,1]$ by the short exact sequence \eqref{ses}.
\end{rmk}

The problem in Example \ref{ginz-ex}, as we noticed in Remark \ref{ginz-rmk}, is that $T_{S_1}$ does not send $\mathcal{P}_\sigma(0,1]$ to $\mathcal{P}_\sigma(\psi,\psi+1]$, a property which is satisfied by any autoequivalence satisfying $\Phi \cdot \sigma = \sigma \cdot g$ for some $g \in \widetilde{GL}^+(2,\mathbb{R})$.

The following lemma shows that, once we have checked the condition \eqref{condition-1}, the condition $\Phi(\mathcal{P}_\sigma(0,1]) = \mathcal{P}_\sigma(\psi,\psi+1]$ for some $\psi \in \mathbb{R}$ is the only thing we need to check to conclude that $\Phi \cdot \sigma = \sigma \cdot g$ for some $g \in \widetilde{GL}^+(2,\mathbb{R})$.

\begin{lem}\label{criterion}
Let $(\Phi,\sigma,M) \in \mathrm{Aut}(\mathcal{D}) \times \mathrm{Stab}_\Gamma^\dagger(\mathcal{D}) \times GL^+(2,\mathbb{R})$.
Assume that $Z_\sigma \circ [\Phi]_\Gamma = M \cdot Z_\sigma$ and that $\Phi(\mathcal{P}_\sigma(0,1]) = \mathcal{P}_\sigma(\psi,\psi+1]$ for some $\psi \in \mathbb{R}$.
Then $\Phi(\mathcal{P}_\sigma(\phi)) = \mathcal{P}_\sigma(f(\phi))$ where $f : \mathbb{R} \to \mathbb{R}$ is the lift of the orientation preserving map $\bar{M} : S^1 \to S^1$ induced by $M$ such that $f(0) = \psi$.
\end{lem}

\begin{proof}
Let $F \in \mathcal{P}_\sigma(\phi)$ be a semistable object for some $\phi \in (0,1]$ and let $E \hookrightarrow \Phi F$ be an injection in $\mathcal{P}_\sigma(\psi,\psi+1]$ with $\frac{1}{\pi}\arg Z_\sigma(E) \eqqcolon \eta$.
Then $\Phi^{-1}E \hookrightarrow F$ is an injection in $\mathcal{P}_\sigma(0,1]$ and $\frac{1}{\pi}\arg Z_\sigma(\Phi^{-1}E) = \bar{M}^{-1}\eta$.
As $F$ is semistable, we get $\bar{M}^{-1}\eta \leq \phi$ and therefore $\eta \leq \bar{M}\phi = \frac{1}{\pi}\arg Z_\sigma(\Phi F)$ proving that $\Phi F$ is semistable of phase $f(\phi)$ (a priori the phase is $\bar{M}\phi+k$ for some $k \in \mathbb{Z}$ but we know that $\Phi F \in \mathcal{P}_\sigma(\psi,\psi+1]$).
Hence we get $\Phi(\mathcal{P}_\sigma(\phi)) \subset \mathcal{P}_\sigma(f(\phi))$.

A similar argument proves the other containment.
\end{proof}

\subsection{Restriction for compatible triples}

We first remark that some special classes of compatible triples $(\Phi,\sigma,g) \in \mathrm{Aut}(\mathcal{D}) \times \mathrm{Stab}_\Gamma^\dagger(\mathcal{D}) \times \widetilde{GL}^+(2,\mathbb{R})$ have been already studied by several authors.

\begin{dfn}[{\cite[Definition 4.1]{DHKK},\cite[Definition 4.6]{Kik2}}]
An autoequivalence $\Phi$ of a triangulated category $\mathcal{D}$ is called {\em pseudo-Anosov} if there exist $\sigma \in \mathrm{Stab}_\Gamma(\mathcal{D})$, $\lambda \in \mathbb{R}$ with $\lvert\lambda\rvert > 1$ and $g = \left(\begin{pmatrix} \lambda^\pm&0\\0&\lambda^\mp \end{pmatrix},f\right) \in \widetilde{GL}^+(2,\mathbb{R})$ such that
\begin{equation*}
\Phi \cdot \sigma = \sigma \cdot g.
\end{equation*}
\end{dfn}

\begin{dfn}[{\cite[Definition 1.1]{Tod}}]
A stability condition $\sigma$ on a triangulated category $\mathcal{D}$ is called of {\em Gepner type} with respect to $(\Phi,\alpha) \in \mathrm{Aut}(\mathcal{D}) \times \mathbb{C}$ if they satisfy
\begin{equation*}
\Phi \cdot \sigma = \sigma \cdot \alpha.
\end{equation*}
\end{dfn}

In the rest of this section, we only consider a numerically finite triangulated category $\mathcal{D}$.
In this case, we will see that sometimes there is a restriction for a triple $(\Phi,\sigma,g) \in \mathrm{Aut}(\mathcal{D}) \times \mathrm{Stab}_\mathcal{N}^\dagger(\mathcal{D}) \times \widetilde{GL}^+(2,\mathbb{R})$ to be compatible.

Let us recall the notion of the categorical volume introduced by Fan--Kanazawa--Yau \cite{FKY} as a categorical analogue of the holomorphic volume.

\begin{dfn}[{\cite[Section 3]{FKY},\cite[Definition 2.11]{Fan3}}]
Let $\mathcal{D}$ be a numerically finite triangulated category.
Fix a basis $\{v_1,\dots,v_k\}$ of $\mathcal{N}(\mathcal{D})$ and let $\chi^{ij}$ be the $(i,j)$-component of the inverse matrix of the matrix $\{\chi(v_i,v_j)\}_{i,j=1}^k$.
Then the {\em categorical volume} of $\sigma = (Z_\sigma,\mathcal{P}_\sigma) \in \mathrm{Stab}_\mathcal{N}(\mathcal{D})$ is defined by
\begin{equation*}
\mathrm{vol}(\sigma) \coloneqq \left| \sum_{i,j=1}^k \chi^{ij} Z_\sigma(v_i)\overline{Z_\sigma(v_j)} \right|.
\end{equation*}
\end{dfn}

\begin{rmk}
The categorical volume does not depend on the choice of a basis $\{v_1,\dots,v_k\}$ of $\mathcal{N}(\mathcal{D})$.
Also note that the categorical volume can be zero \cite[Remark 2.14]{Fan3}.
\end{rmk}

The following lemma describes how the categorical volume changes under the action of $\widetilde{GL}^+(2,\mathbb{R})$ for the odd dimensional Calabi--Yau case.

\begin{lem}\label{vol-lem}
Let $\mathcal{D}$ be a numerically finite $d$-Calabi--Yau triangulated category where $d$ is an odd integer.
Let $\sigma \in \mathrm{Stab}_\mathcal{N}(\mathcal{D})$ and $g \in \widetilde{GL}^+(2,\mathbb{R})$.
Then we have
\begin{equation*}
\mathrm{vol}(\sigma \cdot g) = \det(M_g^{-1})\mathrm{vol}(\sigma).
\end{equation*}
\end{lem}

\begin{proof}
Since $\chi^{ji}=(-1)^d\chi^{ij}=-\chi^{ij}$, we have $\sum_{i,j=1}^k \chi^{ij}Z_\sigma(v_i)Z_\sigma(v_j) = 0$.
Let us write $M_g^{-1} = \begin{pmatrix} a&b\\c&d \end{pmatrix}$.
Then, for every $v \in \mathcal{N}(\mathcal{D})$,
\begin{equation*}
Z_{\sigma \cdot g}(v) = \alpha Z_\sigma(v) + \beta \overline{Z_\sigma(v)}
\end{equation*}
where $\alpha \coloneqq \frac{a+d+i(-b+c)}{2}$ and $\beta \coloneqq \frac{a-d+i(b+c)}{2}$.
Consequently
\begin{align*}
\mathrm{vol}(\sigma \cdot g)
&= \left| \sum_{i,j=1}^k \chi^{ij} (\alpha Z_\sigma(v_i) + \beta \overline{Z_\sigma(v_i)})(\overline{\alpha} \overline{Z_\sigma(v_j)} + \overline{\beta} Z_\sigma(v_j)) \right|\\
&= \left| (\lvert\alpha\rvert^2-\lvert\beta\rvert^2)\sum_{i,j=1}^k \chi^{ij}Z_\sigma(v_i)\overline{Z_\sigma(v_j)} \right|\\
&= (ad-bc) \left| \sum_{i,j=1}^k \chi^{ij}Z_\sigma(v_i)\overline{Z_\sigma(v_j)} \right|\\
&= \det(M_g^{-1})\mathrm{vol}(\sigma)
\end{align*}
as desired.
\end{proof}

For the odd dimensional Calabi--Yau case, we have the following necessary condition for a compatible triple.

\begin{prop}
Let $\mathcal{D}$ be a numerically finite $d$-Calabi--Yau triangulated category where $d$ is an odd integer.
Let $(\Phi,\sigma,g) \in \mathrm{Aut}(\mathcal{D}) \times \mathrm{Stab}_\mathcal{N}^\dagger(\mathcal{D}) \times \widetilde{GL}^+(2,\mathbb{R})$ be a compatible triple such that $\mathrm{vol}(\sigma)>0$.
Then we have $\det(M_g)=1$.
\end{prop}

\begin{proof}
It is easy to see that $\mathrm{vol}(\Phi \cdot \sigma) = \mathrm{vol}(\sigma)$ \cite[Lemma 2.12]{Fan3}.
Therefore, by Lemma \ref{vol-lem}, we get
\begin{equation*}
\mathrm{vol}(\sigma) = \mathrm{vol}(\Phi \cdot \sigma) = \mathrm{vol}(\sigma \cdot g) = \det(M_g^{-1})\mathrm{vol}(\sigma).
\end{equation*}
As $\mathrm{vol}(\sigma)>0$, we conclude that $\det(M_g)=1$.
\end{proof}

\section{Examples}

\subsection{Curves}

Let $X$ be a smooth projective variety and $D^b\mathrm{Coh}(X)$ be the bounded derived category of $X$.
Define the group of {\em standard autoequivalences} of $D^b\mathrm{Coh}(X)$ by
\begin{equation*}
\mathrm{Aut}_\mathrm{std}(D^b\mathrm{Coh}(X)) \coloneqq (\mathrm{Aut}(X) \ltimes \mathrm{Pic}(X)) \times \mathbb{Z}.
\end{equation*}
It is the subgroup of $\mathrm{Aut}(D^b\mathrm{Coh}(X))$ consisting of those autoequivalences of the form $f^*(- \otimes \mathcal{L})[m]$ where $f \in \mathrm{Aut}(X)$, $\mathcal{L} \in \mathrm{Pic}(X)$ and $m \in \mathbb{Z}$.

The Gromov--Yomdin type theorem is known to hold for a smooth projective curve.

\begin{thm}[{\cite[Theorem 3.1 and Proposition 3.3]{Kik1}}]\label{gy-curve}
Let $C$ be a smooth projective curve.
For any $\Phi \in \mathrm{Aut}(D^b\mathrm{Coh}(C))$, we have
\begin{equation*}
h_\mathrm{cat}(\Phi) = \log\rho(\Phi).
\end{equation*}
Furthermore, if $\Phi \in \mathrm{Aut}_\mathrm{std}(D^b\mathrm{Coh}(C))$ then $h_\mathrm{cat}(\Phi)=0$.
\end{thm}

Using Theorem \ref{gy-curve} together with Proposition \ref{linearity} (1) and Corollary \ref{iff-cor} (1), we can show the following.

\begin{prop}\label{curve-prop}
Let $C$ be a smooth projective curve.
For any $\Phi \in \mathrm{Aut}(D^b\mathrm{Coh}(C))$ and $\sigma \in \mathrm{Stab}_\mathcal{N}(D^b\mathrm{Coh}(C))$, we have
\begin{equation*}
h_t(\Phi) = h_{\sigma,t}(\Phi) = \log\rho(\Phi) + \overline{\nu}(\Phi) \cdot t.
\end{equation*}
Furthermore, if $\Phi \in \mathrm{Aut}_\mathrm{std}(D^b\mathrm{Coh}(C))$ then $h_t(\Phi) = h_{\sigma,t}(\Phi) = \overline{\nu}(\Phi) \cdot t$.
\end{prop}

\begin{proof}
Let $g(C)$ be the genus of $C$ and
\begin{equation*}
\sigma_0 \coloneqq (-\mathrm{deg}+i\cdot\mathrm{rk},\mathrm{Coh}(C)) \in \mathrm{Stab}_\mathcal{N}(D^b\mathrm{Coh}(C)).
\end{equation*}
Note that $\sigma_0$ has spanning image.

(1)
If $g(C)>0$ then the action of $\widetilde{GL}^+(2,\mathbb{R})$ on $\mathrm{Stab}_\mathcal{N}(D^b\mathrm{Coh}(C))$ is free and transitive \cite[Theorem 9.1]{Bri1}, \cite[Theorem 2.7]{Mac}.
This implies that $\mathrm{Stab}_\mathcal{N}(D^b\mathrm{Coh}(C))$ is connected and so it is enough to prove the assertion for $\sigma=\sigma_0$ by Lemma \ref{mass-grow-lem}.
Moreover $\mathrm{Stab}_\mathcal{N}(D^b\mathrm{Coh}(C))/\widetilde{GL}^+(2,\mathbb{R})$ consists of one point and thus, for any $\Phi \in \mathrm{Aut}(D^b\mathrm{Coh}(C))$, there exists $g \in \widetilde{GL}^+(2,\mathbb{R})$ such that the triple $(\Phi,\sigma_0,g)$ is compatible (see Remark \ref{triple-rmk}).
The statement now follows from Theorem \ref{gy-curve}, Proposition \ref{linearity} (1) and Corollary \ref{iff-cor} (1).

(2)
If $g(C)=0$ then $\mathrm{Stab}_\mathcal{N}(D^b\mathrm{Coh}(C)) \simeq \mathbb{C}^2$, in particular, it is connected \cite[Theorem 1.1]{Oka}.
Therefore it is enough to prove the assertion for $\sigma=\sigma_0$ by Lemma \ref{mass-grow-lem}.
Moreover, since the anticanonical bundle of $C$ is ample, we have $\mathrm{Aut}(D^b\mathrm{Coh}(C)) = \mathrm{Aut}_\mathrm{std}(D^b\mathrm{Coh}(C))$ \cite[Theorem 3.1]{BO}.
Hence, to conclude the proof, we only have to deal with standard autoequivalences.

For $f \in \mathrm{Aut}(C)$ and $m \in \mathbb{Z}$, we have $f^* \cdot \sigma_0 = \sigma_0$ and $[m] \cdot \sigma_0 = \sigma_0 \cdot ((-1)^m\mathrm{Id},\phi \mapsto \phi+m)$.
Let $\mathcal{L} \in \mathrm{Pic}(C)$ then an explicit computation shows
\begin{equation*}
\begin{pmatrix}
-\mathrm{deg}\,\mathcal{E}\otimes\mathcal{L}\\
\mathrm{rk}\,\mathcal{E}\otimes\mathcal{L}
\end{pmatrix}
=
\begin{pmatrix}
1 & \mathrm{deg}\,\mathcal{L}\\
0 & 1
\end{pmatrix}
\begin{pmatrix}
-\mathrm{deg}\,\mathcal{E}\\
\mathrm{rk}\,\mathcal{E}
\end{pmatrix}.
\end{equation*}
As tensoring with a line bundle preserves $\mathrm{Coh}(C) \subset D^b\mathrm{Coh}(C)$, we can apply Lemma \ref{criterion} and deduce that there exists $h : \mathbb{R} \to \mathbb{R}$ such that
\begin{equation*}
(- \otimes \mathcal{L}) \cdot \sigma_0 = \sigma_0 \cdot \left( \begin{pmatrix}1 & \mathrm{deg}\,\mathcal{L}\\0 & 1\end{pmatrix},h \right).
\end{equation*}

Summing up, we proved that, for any $\Phi \in \mathrm{Aut}_\mathrm{std}(D^b\mathrm{Coh}(C))$, there exists $(M,h) \in \widetilde{GL}^+(2,\mathbb{R})$ such that $(\Phi,\sigma_0,(M,h))$ is compatible and $\rho(M)=1$.
The statement now follows from Theorem \ref{gy-curve}, Proposition \ref{linearity} (1) and Corollary \ref{iff-cor} (1).
\end{proof}

\subsection{Category \texorpdfstring{$D^b\mathrm{Coh}_{(1)}(X)$}{DbCoh(1)(X)}}\label{MP-example}

Let $X$ be an irreducible smooth projective variety of dimension $d \geq 2$.
Then, following \cite{MP}, we define $\mathrm{Coh}^2(X) \subset \mathrm{Coh}(X)$ as the full subcategory of torsion sheaves on $X$ whose support has codimension at least 2.
It is a Serre subcategory of $\mathrm{Coh}(X)$ and so we can take the quotient
\begin{equation*}
\mathrm{Coh}_{(1)}(X) \coloneqq \mathrm{Coh}(X)/\mathrm{Coh}^2(X)
\end{equation*}
which is again an abelian category.

In \cite{MP}, the authors studied the triangulated category $D^b\mathrm{Coh}_{(1)}(X)$ and proved that it controls the birational geometry of $X$.
Moreover they described the space of stability conditions on $D^b\mathrm{Coh}_{(1)}(X)$ proving that it behaves roughly as for the bounded derived category of a curve.
We will now use their description to prove that the analogue between $D^b\mathrm{Coh}_{(1)}(X)$ and the bounded derived category of a curve is also true from a dynamical point of view when $X$ has Picard rank 1.

Let us first recall the description of the space of stability conditions on $D^b\mathrm{Coh}_{(1)}(X)$.
Note that we have an isomorphism $K(D^b\mathrm{Coh}_{(1)}(X)) \cong \mathbb{Z} \oplus \mathrm{Pic}(X)$ given by $\mathrm{rk} \oplus \mathrm{det}$ \cite[Proposition 3.15]{MP}.
Therefore it is natural to consider $\mathrm{Stab}_\Gamma(D^b\mathrm{Coh}_{(1)}(X))$ for $\Gamma \coloneqq \mathbb{Z} \oplus N^1(X)$, i.e., we quotient by numerical equivalence.

In the following, we denote by $N_1(X)$ the group of 1-cycles modulo numerical equivalence, i.e., two cycles $\gamma_1,\gamma_2$ are equivalent if $D \cdot \gamma_1 = D \cdot \gamma$ for any divisor $D$.

\begin{thm}[{\cite[Theorem 4.7]{MP}}]\label{quot-stab}
Let $X$ be an irreducible smooth projective variety of dimension $d \geq 2$.
Then, $\widetilde{GL}^+(2,\mathbb{R})$ acts freely on $\mathrm{Stab}_\Gamma(D^b\mathrm{Coh}_{(1)}(X))$ and any $\widetilde{GL}^+(2,\mathbb{R})$-orbit is a connected component of $\mathrm{Stab}_\Gamma(D^b\mathrm{Coh}_{(1)}(X))$.
The space of connected components is parametrized by the set of rays in the convex cone
\begin{equation*}
C(X) \coloneqq \{ \omega \in N_1(X) \otimes_\mathbb{Z} \mathbb{R} \,|\, \inf \{ \omega \cdot D \,|\, D \text{ is an effective divisor} \} > 0 \}.
\end{equation*}
For each $\omega \in C(X)$, there exists a unique stability condition in the component associated to $\mathbb{R}_{>0}\omega$ of the form $\sigma_\omega \coloneqq (-\omega \cdot c_1 + i \cdot \mathrm{rk},\mathrm{Coh}_{(1)}(X))$.
\end{thm}

Let $\mathrm{Aut}_{(1)}(X)$ be the group of birational automorphisms of $X$ which are isomorphisms in codimension 1.

\begin{thm}[{\cite[Corollary 5.5]{MP}}]\label{MP-thm}
Let $X$ be an irreducible smooth projective variety of dimension $d \geq 2$.
For any $\Phi \in \mathrm{Aut}(D^b\mathrm{Coh}_{(1)}(X))$, there exists a decomposition
\begin{equation*}
\Phi = f^*(- \otimes \mathcal{L})[m]
\end{equation*}
for some $f \in \mathrm{Aut}_{(1)}(X)$, $\mathcal{L} \in \mathrm{Pic}(X)$ and $m \in \mathbb{Z}$.
\end{thm}

Let us denote by $\mathrm{Aut}_0(D^b\mathrm{Coh}_{(1)}(X))$ the subgroup of those autoequivalences such that, for the decomposition as in Theorem \ref{MP-thm}, $\mathcal{L}$ is numerically trivial.
Then every $\Phi \in \mathrm{Aut}_0(D^b\mathrm{Coh}_{(1)}(X))$ preserves the kernel of $K(D^b\mathrm{Coh}_{(1)}(X)) \to \Gamma$ and therefore there is a group homomorphism $[-]_\Gamma : \mathrm{Aut}_0(D^b\mathrm{Coh}_{(1)}(X)) \to \mathrm{Aut}_\mathbb{Z}(\Gamma)$ which makes the diagram
\begin{equation*}
\begin{tikzcd}
K(D^b\mathrm{Coh}_{(1)}(X)) \ar[r,"{[\Phi]}"] \ar[d] & K(D^b\mathrm{Coh}_{(1)}(X)) \ar[d]\\
\Gamma \ar[r,"{[\Phi]}_\Gamma",swap] & \Gamma
\end{tikzcd}
\end{equation*}
commute.

\begin{prop}\label{quot-prop}
Let $X$ be an irreducible smooth projective variety of dimension $d \geq 2$ and Picard rank 1.
Then, for any $\Phi = f^*(- \otimes \mathcal{L})[m] \in \mathrm{Aut}_0(D^b\mathrm{Coh}_{(1)}(X))$ and $\sigma \in \mathrm{Stab}_\Gamma(D^b\mathrm{Coh}_{(1)}(X))$, we have
\begin{equation*}
h_\sigma(\Phi) = \log\rho(\Phi) = \log\rho(f^*|_{N^1(X)}).
\end{equation*}
\end{prop}

\begin{proof}
Theorem \ref{quot-stab} tells us that, in any connected component of $\mathrm{Stab}_\Gamma(D^b\mathrm{Coh}_{(1)}(X))$, there exists a stability condition of the form $\sigma_\omega = (-c_1 \cdot \omega + i \cdot \mathrm{rk},\mathrm{Coh}_{(1)}(X))$ for some $\omega \in C(X)$.
Therefore, by Lemma \ref{mass-grow-lem}, to prove the statement, it is enough to prove it for the stability conditions of the form $\sigma_\omega$.
Notice that such a stability condition has spanning image.

By \cite[Proposition 5.1]{MP}, we have that any $\Phi \in \mathrm{Aut}(D^b\mathrm{Coh}_{(1)}(X))$ preserves $\mathrm{Coh}_{(1)}(X)$ up to shift.
Moreover an explicit computation shows $[m] \cdot \sigma_\omega = \sigma_\omega \cdot ((-1)^m\mathrm{Id},\phi \mapsto \phi+m)$,
\begin{equation*}
Z_{\sigma_\omega} \circ [- \otimes \mathcal{L}]_\Gamma = \begin{pmatrix}1 & - c_1(\mathcal{L}) \cdot \omega\\0 & 1\end{pmatrix} \cdot Z_{\sigma_\omega} \text{ and } Z_{\sigma_\omega} \circ [f^*]_\Gamma = \begin{pmatrix}f^*|_{N^1(X)} & 0\\0 & 1\end{pmatrix} \cdot Z_{\sigma_\omega}.
\end{equation*}
Combining these relations with Lemma \ref{criterion}, we get that, for any $\Phi = f^*(- \otimes \mathcal{L})[m] \in \mathrm{Aut}(D^b\mathrm{Coh}_{(1)}(X))$, we have
\begin{equation*}
\Phi \cdot \sigma_\omega = \sigma_\omega \cdot \left( (-1)^m \begin{pmatrix}f^*|_{N^1(X)} & - c_1(\mathcal{L}) \cdot \omega\\0 & 1\end{pmatrix},h \right)
\end{equation*}
for some $h : \mathbb{R} \to \mathbb{R}$.
 Then the statement follows from Theorem \ref{cat-gy-thm} (1).
\end{proof}

When the Picard rank of $X$ is bigger than 1, things become more complicated because it is not clear how to describe $Z_{\sigma_\omega} \circ [f^*]_\Gamma$.
However, if $X$ is a surface, we can say something more: the closure of $C(X)$ is the nef cone $\mathrm{Nef}(X)$, any $f \in \mathrm{Aut}_{(1)}(X)$ possesses a unique (up to scalar) eigenvector $\alpha$ of eigenvalue $\rho(f^*)$ and $\alpha \in \mathrm{Nef}(X)$ \cite[Theorem 5.1]{DF}.
If $\alpha \in C(X)$ then the same proof as in Proposition \ref{quot-prop} shows $h_{\sigma_\alpha}(f_*(- \otimes \mathcal{L})[m]) = \log\rho(f^*|_{N^1(X)})$ for any numerically trivial $\mathcal{L} \in \mathrm{Pic}(X)$ and $m \in \mathbb{Z}$.

As a byproduct of this last reasoning, we see that if the canonical bundle $\omega_X$ is either ample or antiample then $h_{\sigma_{\omega_X}}(f_*(- \otimes \mathcal{L})[m]) = 0$ since $f^*\omega_X \cong \omega_X$.
When $f \in \mathrm{Aut}(X)$, this was already known because the categorical entropy of $f_*(- \otimes \mathcal{L})[m]$ on $D^b\mathrm{Coh}_{(1)}(X)$ is smaller or equal than that on $D^b\mathrm{Coh}(X)$ and the latter is known to be zero for any autoequivalence of a smooth projective variety with (anti)ample canonical bundle \cite[Theorem 5.7]{KT}.

\section{Further questions}

In this section, we want to propose more a wishful thinking than a precise result.
In \cite{Ike}, the author defined the mass growth $h_{\sigma,t}(\Phi)$ for any $\Phi \in \mathrm{Aut}(\mathcal{D})$ and any $\sigma \in \mathrm{Stab}(\mathcal{D})$.
Moreover he proved that the mass growth only depends on the connected component in which $\sigma$ lies (see Lemma \ref{mass-grow-lem}), that $h_{\sigma,t}(\Phi) \leq h_t(\Phi)$ (see Theorem \ref{mass-grow-thm}) and that we always have $\log\rho(\Phi) \leq h_\sigma(\Phi)$ (see Proposition \ref{yom-mass-grow}).

Most of his constructions can be performed also when $\sigma$ is only a {\em weak stability condition}, i.e., when $Z_\sigma(E) \in \mathbb{R}_{\geq 0}e^{i\pi\phi}$ (resp. $Z_\sigma(E) \in \mathbb{R}_{> 0}e^{i\pi\phi}$) for $0 \neq E \in \mathcal{P}_\sigma(\phi)$ and $\phi \in \mathbb{Z}$ (resp. $\phi \not\in \mathbb{Z}$).
Hence one could try to compute $h_{\sigma,t}(\Phi)$ for a weak stability condition $\sigma$.

While this is a perfectly well-posed question, it is not clear whether the mass growth $h_{\sigma,t}(\Phi)$ is well-behaved when we vary $\sigma$ in the space of weak stability conditions.
It is not even clear whether the inequality $\log\rho(\Phi) \leq h_\sigma(\Phi)$ holds anymore.
For the time being, let us forget these questions and consider $h_{\sigma,t}(\Phi)$ for a weak stability condition $\sigma$.

The reason why we were led to this more general setup was to check whether the fact that the Gromov--Yomdin type theorem holds for standard autoequivalences (see \cite[Theorem 5.7]{KT}) can be traced back to the preservation of a stability condition.

Let $(X,\mathcal{L})$ be a polarized variety of dimension $d$ where $\mathcal{L}$ is a very ample line bundle.
Then $\sigma = (ip_d-p_{d-1},\mathrm{Coh}(X))$ is a weak stability condition on $D^b\mathrm{Coh}(X)$ where $p_d,p_{d-1}$ are the leading coefficients of the Hilbert polynomial \cite[Example 14.5]{BLMNPS}.

\begin{rmk}
If $\omega$ is the class of $\mathcal{L}$ then $p_d(\mathcal{E}) = \mathrm{rk}\,\mathcal{E} \cdot \omega^d$ and $p_{d-1}(\mathcal{E}) = c_1(\mathcal{E}) \cdot \omega^{d-1} + \mathrm{rk}\,\mathcal{E} \cdot \mathrm{td}(X)_1 \cdot \omega^{d-1}$ for any vector bundle $\mathcal{E}$ where $\mathrm{td}(X)_1 \in H^2(X,\mathbb{C})$ is the degree two component of the Todd class of $X$.
\end{rmk}

Now let $\Phi = f^*(- \otimes \mathcal{L}')[m]$ be a standard autoequivalence.
Then $\chi(X,f^*\mathcal{E}) = \chi(X,\mathcal{E})$ for any $\mathcal{E} \in D^b\mathrm{Coh}(X)$ and therefore $Z_\sigma \circ [f^*] = Z_\sigma$.
For a line bundle $\mathcal{L}'$, we have
\begin{align*}
Z_\sigma(\mathcal{E}\otimes\mathcal{L}')
&= ip_d(\mathcal{E}\otimes\mathcal{L}') - p_{d-1}(\mathcal{E}\otimes\mathcal{L}')\\
&= ip_d(\mathcal{E}) - p_{d-1}(\mathcal{E}) - \mathrm{rk}\,\mathcal{E} \cdot c_1(\mathcal{L}') \cdot \omega^{d-1}\\
&= \begin{pmatrix}1&\frac{c_1(\mathcal{L}') \cdot \omega^{d-1}}{\omega^d}\\0&1\end{pmatrix} \cdot Z_\sigma(\mathcal{E}).
\end{align*}

As $\Phi$ shifts the subcategory $\mathrm{Coh}(X)$ we can apply Lemma \ref{criterion}, and we obtain that the assumption of Theorem \ref{cat-gy-thm} holds.
Hence we get $h_\sigma(\Phi) = \log\rho(\Phi)$.
Furthermore, if the canonical bundle $\omega_X$ is ample (resp. antiample), we can take $\mathcal{L} = \omega_X$ (resp. $\mathcal{L} = \omega_X^\vee$).
Then $f^*\omega_X \cong \omega_X$ (resp. $f^*\omega_X^\vee \cong \omega_X^\vee$) and we have $h_\sigma(\Phi) = \log\rho(\Phi) = 0$ for any $\Phi \in \mathrm{Aut}_\mathrm{std}(D^b\mathrm{Coh}(X)) = \mathrm{Aut}(D^b\mathrm{Coh}(X))$.


\begin{thebibliography}{99}

\bibitem{BK1} F. Barbacovi and J. Kim, {\em Entropy of the composition of two spherical twists}, arXiv:2107.06079.
\bibitem{BLMNPS} A. Bayer, M. Lahoz, E. Macr\`i, H. Nuer, A. Perry and P. Stellari, {\em Stability conditions in families}, Publ. Math. Inst. Hautes \'Etudes Sci. {\bf 133} (2021) 157--325.
\bibitem{BK2} A. Bondal, M. Kapranov, {\em Representable functors, Serre functors, and mutations}, Math. USSR-Izv. {\bf 35}(3) (1990) 519--541.
\bibitem{BO} A. Bondal and D. Orlov, {\em Reconstruction of a variety from the derived category and groups of autoequivalences}, Compositio Math. {\bf 125}(3) (2001) 327--344.
\bibitem{Bri1} T. Bridgeland, {\em Stability conditions on triangulated categories}, Ann. Math. {\bf 166} (2007) 317--345.
\bibitem{Bri2} T. Bridgeland, {\em Spaces of stability conditions}, Proc. Sympos. Pure Math. {\bf 80.1} (2009) 1--21.
\bibitem{CPR} S. Cantat and O. Paris-Romaskevich, {\em Automorphisms of compact Kähler manifolds with slow dynamics}, Trans. Amer. Math. Soc. {\bf 374}(2) (2021) 1351--1389.
\bibitem{DF} J. Diller and C. Favre, {\em Dynamics of bimeromorphic maps of surfaces}, Amer. J. Math. {\bf 123}(6) (2001) 1135--1169.
\bibitem{DHKK} G. Dimitrov, F. Haiden, L. Katzarkov and M. Kontsevich, {\em Dynamical systems and categories}, Contemp. Math. {\bf 621}, 133--170, AMS, 2014.
\bibitem{EL} A.D. Elagin and V.A. Lunts, {\em Three notions of dimension for triangulated categories}, arXiv:1901.09461.
\bibitem{Fan1} Y.-W. Fan, {\em Entropy of an autoequivalence on Calabi--Yau manifolds}, Math. Res. Lett. {\bf 25}(2) (2018) 509--519.
\bibitem{Fan2} Y.-W. Fan, {\em On entropy of $\mathbb{P}$-twists}, arXiv:1801.10485.
\bibitem{Fan3} Y.-W. Fan, {\em Systolic inequalities for K3 surfaces via stability conditions}, Math. Z. (2021).
\bibitem{FF} Y.-W. Fan and S. Filip, {\em Asymptotic shifting numbers in triangulated categories}, arXiv:2008.06159.
\bibitem{FFHKL} Y.-W. Fan, S. Filip, F. Haiden, L. Katzarkov and Y. Liu, {\em On pseudo-Anosov autoequivalences}, arXiv:1910.12350.
\bibitem{FFO} Y.-W. Fan, L. Fu and G. Ouchi, {\em Categorical polynomial entropy}, Adv. Math. {\bf 383} (2021) 107655.
\bibitem{FKY} Y.-W. Fan, A. Kanazawa and S.-T. Yau, {\em Weil--Petersson geometry on the space of Bridgeland stability conditions}, Comm. Anal. Geom. {\bf 29}(3) (2021) 681--706.
\bibitem{Gin} V. Ginzburg, {\em Calabi--Yau algebras}, arXiv:math/0612139.
\bibitem{Gro} M. Gromov, {\em On the entropy of holomorphic maps}, Enseign. Math. (2) {\bf 49} (2003) 217--235.
\bibitem{Ike} A. Ikeda, {\em Mass growth of objects and categorical entropy}, to appear in Nagoya Math. J., arXiv:1612.00995.
\bibitem{Joy} D. Joyce, {\em Conjectures on Bridgeland stability for Fukaya categories of Calabi--Yau manifolds, special Lagrangians, and Lagrangian mean curvature flow}, EMS Surv. Math. Sci. {\bf 2}(1) (2015) 1--62.
\bibitem{KY} B. Keller and D. Yang, {\em Derived equivalences from mutations of quivers with potential}, Adv. Math. {\bf 226} (2011) 2118--2168.
\bibitem{Kik1} K. Kikuta, {\em On entropy for autoequivalences of the derived category of curves}, Adv. Math. {\bf 308} (2017) 699--712.
\bibitem{Kik2} K. Kikuta, {\em Curvature of the space of stability conditions}, arXiv:1907.10973.
\bibitem{KST} K. Kikuta, Y. Shiraishi and A. Takahashi, {\em A note on entropy of auto-equivalences: Lower bound and the case of orbifold projective lines}, Nagoya Math. J. {\bf 238} (2020) 86--103.
\bibitem{KT} K. Kikuta and A. Takahashi, {\em On the categorical entropy and the topological entropy}, Int. Math. Res. Not. {\bf 2019}(2) (2017) 457--469.
\bibitem{Kim} J. Kim, {\em Computation of categorical entropy via spherical functors}, arXiv:2102.08590.
\bibitem{Kon} M. Kontsevich, {\em Homological algebra of mirror symmetry}, Proceedings of the International Congress of Mathematicians (Z\"urich, 1994), 120--139, Birkh\"auser, 1995.
\bibitem{KP} A. Kuznetsov and A. Perry, {\em Serre functors and dimensions of residual categories}, arXiv:2109.02026.
\bibitem{Mac} E. Macr\`i, {\em Stability conditions on curves}, Math. Res. Lett. {\bf 14}(4) (2007) 657--672.
\bibitem{Mat} D. Mattei, {\em Categorical vs topological entropy of autoequivalences of surfaces}, arXiv:1909.02758.
\bibitem{MP} S. Meinhardt and H. Partsch, {\em Quotient categories, stability conditions, and birational geometry}, Geom. Dedicata {\bf 173} (2014) 365--392.
\bibitem{Oka} S. Okada, {\em Stability manifold of $\mathbb{P}^1$}, J. Alg. Geom. {\bf 15}(3) (2006) 487--505.
\bibitem{Ouc} G. Ouchi, {\em On entropy of spherical twists}, Proc. Amer. Math. Soc. {\bf 148}(3) (2020) 1003--1014.
\bibitem{ST} P. Seidel and R.P. Thomas, {\em Braid group actions on derived categories of coherent sheaves}, Duke Math. J. {\bf 108}(1) (2001) 37--108.
\bibitem{Tod} Y. Toda, {\em Gepner type stability conditions on graded matrix factorizations}, Alg. Geom. {\bf 1}(5) (2014) 613--665.
\bibitem{Woo1} J. Woolf, {\em Some metric properties of spaces of stability conditions}, Bull. London Math. Soc. {\bf 44} (2012) 1274--1284.
\bibitem{Woo2} J. Woolf, {\em Mass-growth of triangulated auto-equivalences}, arXiv:2109.13163.
\bibitem{Yom} Y. Yomdin, {\em Volume growth and entropy}, Israel J. Math. {\bf 57}(3) (1987) 285--300.
\bibitem{Yos} K. Yoshioka, {\em Categorical entropy for Fourier--Mukai transforms on generic abelian surfaces}, J. Algebra {\bf 556} (2020) 448--466.

\end{thebibliography}
\end{document}